\documentclass[a4paper,10pt]{article}
\usepackage[utf8]{inputenc}

\usepackage{amsfonts}
\usepackage{color,enumerate,wrapfig,amssymb}
\usepackage{amsmath}
\usepackage{amsthm}
\usepackage{esint}
\usepackage{mathrsfs}
\usepackage{textcomp}
\usepackage{authblk}

\usepackage[english]{babel}

\theoremstyle{plain}
\newtheorem{theorem}{Theorem}
\newtheorem{lemma}[theorem]{Lemma}
\newtheorem{proposition}[theorem]{Proposition}
\newtheorem{corollary}[theorem]{Corollary}

\newtheorem{example}[theorem]{Example}
\newtheorem{definition}[theorem]{Definition}

\numberwithin{equation}{section}

\numberwithin{theorem}{section}

\newcommand{\dist}{dist}
\title{Regularity of the free boundary in the biharmonic obstacle problem}

\author{Gohar Aleksanyan}

\begin{document}

\maketitle

\begin{abstract}
In this article we use flatness improvement argument to study the regularity of the free boundary 
for the biharmonic obstacle problem with zero obstacle. Assuming that the solution is almost one-dimensional,
and that the non-coincidence set is an non-tangentially
accessible (NTA) domain,
we derive the $C^{1,\alpha}$-regularity of the free boundary in a small ball centered at the origin.

\par
From the $C^{1,\alpha}$-regularity of the free boundary we conclude that the solution to the biharmonic obstacle problem
is locally $ C^{3,\alpha}$ up to the free boundary, and therefore $C^{2,1} $.
In the end we study an example, showing that in general $ C^{2,\frac{1}{2}}$ is the best regularity that a 
solution may achieve in dimension $n \geq 2$.

\end{abstract}

\tableofcontents

\section{Introduction}

Let $ \Omega \subset \mathbb{R}^n$ be a given domain,  and  $\varphi \in C^2 (\overline{\Omega}), \varphi \leq 0
\textrm{ on } \partial \Omega$  be a given function, called an obstacle. Then the minimiser to the following functional 
\begin{equation} \label{J}
 J[u]= \int_\Omega \left(\Delta u(x) \right)^2 dx,
\end{equation}
over all functions $ u \in W_0^{2,2}(\Omega),$ such that $ u\geq \varphi$,  is called the solution to the biharmonic
obstacle problem with obstacle $\varphi$. The solution
satisfies the following variational
inequality
\begin{equation*}
\Delta^2 u\geq 0, ~ u \geq \varphi, ~ \Delta^2u \cdot (u-\varphi)=0.
\end{equation*}
\par
It has been shown in  \cite{Fr}  and \cite{Fr71} that the solution $u \in W^{3,2}_{loc}(\Omega)$, 
$ \Delta u \in L^\infty_{loc}(\Omega) $, and moreover $ u\in W^{2,\infty}_{loc}(\Omega)$, see also \cite{CF}. Furthermore,
in the paper \cite{CF},
the authors show that in dimension $n=2 $ the solution $ u\in C^2(\Omega)$ and that the free boundary 
$  \Gamma_u :=\partial \{ u= \varphi\}$ lies on a $C^1$-curve in a neighbourhood of the points
$x_0\in \Gamma_u$, such that $ \Delta u(x_0)> \Delta \varphi(x_0)$. 

\par
The setting of our problem is slightly different from the one in \cite{CF}, \cite{Fr71} and
\cite{Fr}. We consider a zero-obstacle problem with general nonzero boundary conditions. 
Let $ \Omega$ be a bounded domain in $ \mathbb{R}^n $ with smooth boundary. We consider the
problem of minimising the functional \eqref{J}
over the admissible set
\begin{equation*}
\mathscr{A}:=\left\lbrace u \in W^{2,2}(\Omega),~ u \geq 0, ~u=g>0, \frac{\partial u}{\partial \nu}=f 
\textrm{ on } \partial \Omega \right\rbrace .
\end{equation*}
The minimiser $u$ exists, it is unique. The minimiser is called the solution to the biharmonic obstacle problem. 
We will denote the free boundary by $ \Gamma_u := \partial \Omega_u \cap \Omega$, where $\Omega_u: = \{ u > 0\}$.

\par  
There are several important questions regarding the biharmonic obstacle problem that remain open. For example,
the optimal regularity of the solution, the characterisation of blow-ups at free boundary points, etc.
In this article we focus on the regularity of the free boundary for an $n$-dimensional biharmonic obstacle problem,
assuming that the solution is close to the one-dimensional solution $\frac{1}{6}(x_n)_+^3$. In \cite{ALS}, using flatness 
improvement argument, the author, John Andersson,  shows that
the free boundary in the $p$-harmonic obstacle problem is a $C^{1,\alpha}$ graph in a neighbourhood of the points where
the solution is almost one-dimensional.
We apply the same technique in order to study 
the regularity of the free boundary in the biharmonic obstacle problem.

\par
In Section 2
we study the basic properties of the solution in the new setting, and show that it is locally in 
$W^{3,2} \cap C^{1,\alpha}$. The material in this section is essentially known, and it has been adjusted to the setting of our problem.

\par
In Section 3 we introduce the class $ \mathscr{B}_\kappa^\varrho(\varepsilon)$ of solutions to the biharmonic obstacle problem, 
that are close to the one-dimensional solution
$  \frac{1}{6} (x_n)_+^3$.
Following \cite{ALS}, we show 
that if $\varepsilon$ is small enough, then there exists a rescaling $u_s(x)= \frac{u(sx)}{s^3}$, such that
\begin{equation*}
\lVert \nabla'u_s \rVert_{W^{2,2}(B_2)} \leq \gamma \lVert \nabla'u \rVert_{W^{2,2}(B_2)} \leq  \gamma \varepsilon 
\end{equation*}
in a normalised coordinate system, where
$ \nabla'_\eta := \nabla -\eta (\eta \cdot \nabla), \nabla':= \nabla'_{e_n}$, and
$\gamma <1$ is a constant. Repeating the argument for the rescaled solutions,
$ u_{s^k}$, we show that there exists a unit vector $ \eta_0 \in \mathbb{R}^n$, such that 
\begin{equation}
\frac{ \lVert \nabla'_{\eta_0} u_{s^k} \rVert_{W^{2,2}(B_2)}}{\Arrowvert D^3 u_{s^k}\rVert_{L^2(B_1)}} \leq \gamma^k \varepsilon  
\end{equation}
for some $ 0<s<\gamma<1$.
Then the $C^{1,\alpha}$-regularity of the free boundary in a neighbourhood of the origin follows
via a standard iteration argument.

\par
From the $C^{1,\alpha}$-regularity of the free boundary it follows that $ \Delta u \in C^{1,\alpha}$ up to the free 
boundary.
We move further and show that $ u $ is $  C^{3,\alpha} $ up to the free boundary.
Thus a solution $ u\in \mathscr{B}_\kappa^\varrho(\varepsilon)$ is locally $C^{2,1}$, which is 
the best regularity that a solution may achieve.
We provide a two-dimensional counterexample to the $C^{2,1}$-regularity, showing that 
without our flatness assumptions there
exists a  solution that is
$ C^{2,\frac{1}{2}}$ but is not  $ C^{2,\alpha}$ for $  \alpha >\frac{1}{2} $. Hence $ C^{2,\frac{1}{2}}$ is the best regularity that a solution may 
achieve in dimension $ n \geq 2$.

\subsection*{Acknowledgements}

First of all I  wish to thank my advisor John Andersson for suggesting this research project and for
all his help and encouragement.  \\
I am grateful to Erik Lindgren for reading preliminary versions of the manuscript and thereby improving it significantly. \\
I would like to thank the referee for pointing out a gap in one of the proofs, which has been corrected in the current version.

\section{The obstacle problem for the biharmonic operator}

In this section we study the basic properties of the solution 
 to the biharmonic obstacle problem with zero obstacle. Most of the material in this section is known for the biharmonic obstacle problem with a general obstacle, and zero boundary conditions. We need to have quantitive estimates for the analysis of our problem, therefore the proofs are included.

 \par
 First the existence and uniqueness of the solution to the biharmonic obstacle problem are shown.
 Then we show that the solution is in the space
 $ W_{loc}^{3,2} \cap C_{loc}^{1, \alpha}$. 

\subsection{Existence, uniqueness and $ W^{3,2}$-regularity of the solution}

Let us start with the proof of existence and uniqueness of the minimiser of functional \eqref{J}.
Throughout the discussion we denote by $ B_R(x_0) $ the open ball in $ \mathbb{R}^n$, centered
at $ x_0 \in \mathbb{R}^n$, with radius $R>0$, and $ B_R:= B_R(0)$, $ B_R^+:= \{x_n >0\} \cap B_R$.

\begin{lemma} \label{1}
Let $\Omega $ be an open bounded subset of $\mathbb{R}^n$ with a smooth boundary.
Then the functional \eqref{J} admits a unique minimiser in the set $ \mathscr{A} $.
\end{lemma}

\begin{proof}
Here  we use the standard terminology from  \cite{evans}.
Let us start with an observation that the functional $ J $ is weakly lower semicontinuous,
i.e. given a sequence $\{ u_k \}$  converging weakly to a function $u \in W^{2,2}(\Omega) $, then
\begin{equation} \label{lsc}
 \liminf_{k\rightarrow \infty} J[u_k] \geq J[u]. 
\end{equation}

\par
Upon passing to a subsequence, we may assume that
\begin{equation*} 
\liminf_{k\rightarrow \infty} J[u_k] = \lim_{k\rightarrow \infty} J[u_k].
\end{equation*}
According to the definition of weak convergence in $W^{2,2}(\Omega)$,
$ \Delta u_k $  converges to $\Delta u $ weakly in $L^{2}(\Omega) $, hence
\begin{equation*}
\lim_{k\rightarrow \infty} {\int \Delta u_k \Delta u} = \int (\Delta u )^2,
\end{equation*}
and
the inequality
\begin{equation*}
\int (\Delta u )^2+ \int (\Delta u_k)^2-2\int \Delta u_k \Delta u =\int ( \Delta u_k -\Delta u )^2\geq 0
\end{equation*}
implies
\begin{equation*}
 \int (\Delta u_k)^2 \geq 2\int \Delta u_k \Delta u-\int (\Delta u )^2,
\end{equation*}
after passing to a limit as $k \rightarrow \infty$, we get the desired inequality, \eqref{lsc}.

\par
Next we take a minimising sequence  $\{ u_k \} \subset \mathscr{A} $, and show that it converges weakly
to some function $u$ in $W^{2,2}(\Omega)$ through a subsequence, and that $u$ is an admissible function.
Define 
\begin{equation*}
m:=\inf_{v \in \mathscr{A}} \int (\Delta v)^2,
\end{equation*}
then
\begin{equation*}
\lim_{k\rightarrow \infty} J[u_k]= m .
\end{equation*}
Let us note that $ J[u_k]={\Arrowvert  \Delta u_k  \Arrowvert}^2_{L^2}$, so $\Delta u_k$
is bounded in $L^2$, and since $u_k-\omega =0$ and
$\frac{\partial (u_k-\omega)}{\partial n }=0 $ on $ \partial \Omega $ in the trace sense
for any fixed $ \omega \in \mathscr{A} $, the sequence is
bounded in $W^{2,2}(\Omega)$. Hence it has a subsequence which converges weakly in $W^{2,2}$,
we will keep the notation, call it $\{ u_k \} $. We want to show that the limit function
$ u \in \mathscr{A}$.
According to the Sobolev embedding theorem 
 $\{ u_k \}$ converges to $u $ strongly in $L^2$ up to a subsequence, hence upon passing to a new subsequence
 $ u_k \rightarrow u $
a.e. in $ \Omega $. The latter proves that $u \geq 0$ a.e..

\par
It remains to show that $u$ satisfies the boundary conditions.
For any $\omega \in \mathscr{A}$,  $u_k - \omega \in W_0^{2,2}(\Omega)$, since
$W_0^{2,2}(\Omega)$ is a closed, linear subspace of $W^{2,2}(\Omega)$, it is weakly closed, according to
Mazur's theorem (\cite{evans}, pp. 471 and 723). This proves that $u - \omega \in W_0^{2,2}(\Omega)$ and therefore
$ u \in \mathscr{A} $.

\par
According to \eqref{lsc}, 
$ m \geq J[u]$, but the reversed inequality is also true since
$u$ is admissible and according to our choice of the sequence $\{ u_k \}$. Thus $m= J[u]$, and
$u$ is a minimiser.

\par
The uniqueness of the minimiser follows from the convexity of the functional:
assuming that both $u$  and $v$ are minimisers, it follows that $\frac{u+v}{2}$ is also admissible, so
\begin{equation*}
 J\left[ \frac{u+v}{2}\right] \geq \frac{J[u]+J[v]}{2},
\end{equation*}
 but the reversed inequality is also true
with equality if and only if $\Delta u=\Delta v $.
Thus if $ u$ and $v$ are both minimisers in $\mathscr{A}$ then 
 $\Delta (u- v)=0 $ and   $u-v \in W_0^{2,2}(\Omega)$, which implies that  $u=v$ in $\Omega$.

\end{proof}

\par
Now we turn our attention to the regularity of the solution to the biharmonic obstacle problem.

\begin{proposition} \label{p1}
 Let $u$ be the solution to the biharmonic obstacle problem in the unit ball $B_1$, then
\begin{equation*}
\Arrowvert \Delta u \rVert_{W^{1,2}(B_{\frac{1}{2}})}
\leq C \Arrowvert u \rVert_{W^{2,2}(B_1)},
\end{equation*}
where the constant $C $ depends only on the space dimension.
\end{proposition}

\begin{proof}
The proof is based on a difference quotient method. 
Let $ \{e_1,e_2, ..., e_n\}$ be the standard basis in $\mathbb{R}^n $.
For a fixed $i \in \{1,2,...,n\}$
denote 
\begin{equation}
u_{i, h}(x):= u(x+ h e_i),~ \textrm{ for } 
 x \in B_{1-h}.
 \end{equation}

\par
Take a nonnegative function $ \zeta \in C_0^\infty (B_{\frac{3}{4}}) $, 
such that $ \zeta \equiv 1$ in $ B_{\frac{1}{2}}$.
Then for small values of the parameter $t>0$, the function 
$ u + t \zeta^2 (u_{i, h}-u)$ is admissible for the biharmonic obstacle problem
in $B_1$. Indeed, $ u + t \zeta^2 (u_{i, h}-u) = u(1-t\zeta ^2)+t \zeta^2 u_{i,h} \geq 0$
if $t>0$ is small, and obviously it satisfies the same boundary
conditions as the minimiser $u$. Hence 
\begin{equation}\label{t1}
 \int_{B_1}\left(\Delta (u + t \zeta ^2(u_{i, h}-u))\right)^2 \geq \int_{B_1}(\Delta u)^2.
\end{equation}
Assuming that $h< \frac{1}{4}$, the inequality will still hold if we 
replace the integration over the ball $B_1$ by $B_{1-h}$, since $\zeta$
is zero outside the ball $B_{\frac{3}{4}}$.

\par
It is clear that $u_{i, h}$ is the solution to the biharmonic obstacle problem
in $B_{1-h}$, and $ u_{i, h} + t \zeta^2 (u-u_{i, h})$ is an admissible function.
Hence
\begin{equation} \label{t2}
 \int_{B_{1-h}} \left(\Delta (u_{i, h} + t \zeta^2 (u-u_{i, h}))\right)^2 \geq \int_{B_{1-h}}(\Delta u_{ i, h})^2.
\end{equation}
After dividing both sides of the inequalities \eqref{t1} and \eqref{t2} by $t$, and taking the limit
as $t\rightarrow 0$, we get
\begin{equation} \label{h1}
  \int_{B_{1-h}} \Delta u \Delta \left(\zeta^2(u_{i, h} -u)\right) \geq 0,
\end{equation}
and
\begin{equation} \label{h2}
  \int_{B_{1-h}} \Delta u_{i, h} \Delta \left(\zeta^2(u -u_{i, h})\right) \geq 0.
\end{equation}
We rewrite inequalities \eqref{h1} and \eqref{h2} explicitly, that is
 \begin{equation*}
  \int_{B_{1-h}} \Delta u \left( (u_{i, h}-u)\Delta \zeta^2 + \zeta ^2 \Delta (u_{i, h}-u)+ 
  2 \nabla \zeta^2 \nabla (u_{i, h}-u) \right) \geq 0, \textrm{ and }
\end{equation*}
\begin{equation*}
  \int_{B_{1-h}} \Delta u_{i, h} \left( (u-u_{i, h})\Delta \zeta^2 + \zeta^2 \Delta (u-u_{i, h})+ 
  2 \nabla \zeta^2 \nabla (u-u_{i, h}) \right) \geq 0.
\end{equation*}
After summing the inequalities above, we obtain
\begin{equation*}
  \int_{B_{1-h}} \zeta^2 (\Delta (u_{ i, h}-u))^2 \leq 
\end{equation*}
\begin{equation*}
 \int_{B_{1-h}} (u_{i, h}-u) \Delta \zeta^2 \Delta (u-u_{i, h})+
 4 \int_{B_{1-h}} \nabla \zeta \nabla (u_{i, h}-u) \zeta \Delta (u-u_{i, h}).
\end{equation*}
 Dividing both sides of the last inequality by ${ h}^2$, we get
\begin{equation} \label{eq2}
\begin{aligned}
 \int_{B_{1-h}} \frac{\zeta^2(\Delta u_{i,h}-\Delta u)^2}{ h^2} \leq 
 \int_{B_{1-h}} \frac{(u_{i, h}-u)}{ h^2} \Delta \zeta^2 \Delta (u-u_{i,h}) \\
 + 4 \int_{B_{1-h}} \nabla \zeta \frac{(\nabla u_{i, h}-\nabla u )}{ h}
 \zeta \frac{ ( \Delta u-\Delta u_{i, h} )}{ h} .
 \end{aligned}
\end{equation}
First let us study the first integral on the right side of \eqref{eq2}
\begin{equation} \label{one}
\begin{aligned}
 \left| \int_{B_{1-h}} \frac{(u_{i, h}-u)}{{h}^2} \Delta \zeta^2 \Delta (u-u_{i, h}) \right| \\
=\left| \int_{B_{1-h}} \Delta u \left( \frac{(u_{i, h}-u)}{ h^2} \Delta \zeta^2 
-\frac{(u-u_{i,- h})}{ h^2} \Delta \zeta^2_{i, - h} \right) \right| \\
=\left| { \int_{B_{1-h}} \Delta u \Delta \zeta^2 \left( \frac{u_{i, h}-2u+u_{i,- h}}{ h^2} \right)+
\Delta u \left(\frac{\Delta \zeta^2-\Delta \zeta^2_{i, h}}{ h}\right) 
\left (\frac{u-u_{i,- h}}{ h}\right) } \right| \\
\leq C \Arrowvert \Delta u \rVert_{L^2(B_1)} \Arrowvert u \rVert_{W^{2,2}(B_1)},
\end{aligned}
\end{equation}
where we applied H\"{o}lder's inequality, and used the fact that the $L^2$-norm of the first and second 
order difference quotients of 
a function $ u \in W^{2,2}$ are uniformly bounded by its $ W^{2,2}$-norm.

\par
Next we estimate the absolute value of the second integral in \eqref{eq2}
\begin{equation} \label{two}
\begin{aligned} 
\left| \int_{B_{1-h}} \nabla \zeta \left(\frac{\nabla u_{i, h}-\nabla u}{ h} \right)
\zeta \left (\frac{\Delta u-\Delta u_{i, h}}{ h} \right) \right|  \\
\leq 8 \int_{B_{1-h}} \lvert\nabla \zeta  \rvert^2 \frac{\lvert \nabla u_{i, h}-\nabla u \rvert^2 }{ h^2} 
+ \frac{1}{8} \int_{B_{1-h}} \zeta^2 \frac{(\Delta (u_{i, h}-u))^2}{ h^2} ,
\end{aligned}
\end{equation}
where we applied Cauchy's inequality.

\par
Combining inequalities \eqref{eq2}, \eqref{one} and \eqref{two}, we obtain
\begin{equation*} 
\int_{B_{1-h}} \zeta^2 \frac{(\Delta (u_{i,h}-u))^2}{ h^2} \leq C \Arrowvert u \rVert^2_{W^{2,2}(B_1)}.
\end{equation*}

\par
According to our choice of function $\zeta$,
\begin{equation*}
  \int_{B_{\frac{1}{2}}} \frac{(\Delta(u_{i, h}-u))^2}{h^2}
  \leq  \int_{B_{1-h}} \zeta^2 \frac{(\Delta (u_{i, h}-u))^2}{h^2},
\end{equation*}
so the $L^2$-norm of the difference quotients of $\Delta u$ is uniformly 
bounded in $B_{\frac{1}{2}}$ hence 
$\Delta u \in W^{1,2}(B_{\frac{1}{2}}) $, and
\begin{equation*}
 \Arrowvert \Delta u \rVert_{W^{1,2}(B_{\frac{1}{2}})}
\leq C \Arrowvert u \rVert_{W^{2,2}(B_1)},
\end{equation*}
where the constant $C$ depends only on the function $ \zeta$, and can be computed
explicitly, depending only on the space dimension.
 
\end{proof}

\begin{corollary} \label{cor23}
 Assume that $\Omega$ is a bounded  open set in $ \mathbb{R}^n$. Then the solution 
 to the obstacle problem is in $W^{3,2}(K)$ for any $K \subset \subset \Omega$, and
\begin{equation}\label{w32K}
 \Arrowvert u \rVert_{W^{3,2}(K)}
\leq C \Arrowvert u \rVert_{W^{2,2}(\Omega)},
\end{equation}
where the constant $C$ depends on the space dimension $n $ and on  $ \dist(K, \partial \Omega)$.
\end{corollary}

\begin{proof}
It follows from Proposition \ref{p1} by a standard covering argument that
\begin{equation*}
 \Arrowvert \Delta u \rVert_{W^{1,2}(\Omega')} \leq C_{\Omega'} \Arrowvert u \rVert_{W^{2,2}(\Omega)},
\end{equation*}
for any $ \Omega' \subset \subset  \Omega$.

\par
Let $ K \subset \subset  \Omega' \subset \subset \Omega $,
 according to the Calder\'{o}n-Zygmund inequality (\cite{GT}, Theorem 9.11),
 \begin{equation*}
  \Arrowvert D^3 u\rVert_{L^{2}(K)} \leq C_K \left( \Arrowvert \Delta u\rVert_{W^{1,2}(\Omega')}+
 \Arrowvert u \rVert_{W^{2,2}(\Omega')}\right).
 \end{equation*}
 Then it follows that $u$ is in $W^{3,2}$ locally, with the estimate \eqref{w32K}.
\end{proof}

\begin{lemma} \label{zeta}
Let $ u $ be the solution to the biharmonic obstacle problem in $\Omega$. Take 
 $  K \subset \subset \Omega$, and a function $ \zeta \in C^\infty_0(K)$, $ \zeta \geq 0$, 
 then
 \begin{equation}
  \int_\Omega \Delta u_{x_i} \Delta ( \zeta u_{x_i}) \leq 0,
 \end{equation}
for all $ i=1,2,..., n$.
\end{lemma}

\begin{proof}
 Fix $1 \leq  i \leq n$, denote $ u_{i,h}(x):= u(x+he_i)$, where $0 < \lvert h \rvert < \dist(K,\partial \Omega)$,
 hence $u_{i,h} $ is defined in $K$. Let us observe that 
 the function $ u+ t \zeta (u_{i,h}- u )$ is well defined and  nonnegative in $ \Omega$
 for any $ 0<t< \frac{1}{ \lVert \zeta \rVert_{L^\infty}}$, and it satisfies the same boundary conditions as $u$.
 Therefore
 \begin{equation*}
  \int_\Omega \left(  \Delta (u+ t \zeta(u_{i,h}- u  )) \right)^2 \geq \int_\Omega (\Delta u )^2,
 \end{equation*}
after dividing the last inequality by $ t $, and taking the limit as $ t \rightarrow 0 $, we obtain
\begin{equation} \label{ih1}
 \int_K \Delta u \Delta (\zeta ( u_{i,h}-u )) \geq 0.
\end{equation}
Note that $ u_{i,h}$ is the solution to the biharmonic obstacle problem in $ K$, and 
$ u_{i,h}+ t \zeta (u-u_{i,h})$ is an admissible function, hence 
 \begin{equation*}
  \int_K  \left(  \Delta (u_{i,h}+ t \zeta(u-u_{i,h}  )) \right)^2 \geq \int_K (\Delta u_{i,h })^2,
 \end{equation*}
 after dividing the last inequality by $ t $, and taking the limit as $ t \rightarrow 0 $, we obtain
\begin{equation} \label{ih2}
 \int_K \Delta u_{i,h} \Delta (\zeta ( u- u_{i,h} )) \geq 0.
\end{equation}

\par
Inequalities \eqref{ih1} and \eqref{ih2} imply that 
\begin{equation} \label{xi}
\int_K  (\Delta u_{i, h}- \Delta u ) \Delta ( \zeta ( u_{i,h}-u)) \leq 0, 
\end{equation}
dividing the last inequality by $ h^2$, and taking into account that $ u\in W^{3,2}_{loc}$, we may pass to 
the limit as $ \lvert h \rvert \rightarrow 0$ in \eqref{xi}, and conclude that
\begin{equation*}
 \int_K \Delta u_{x_i} \Delta(\zeta u_{x_i}) \leq 0.
\end{equation*}

\end{proof}

\subsection{$ C^{1,\alpha}$-regularity of the solution}

It has been shown in \cite{CF}, Theorem 3.1 that $ \Delta u \in L^\infty_{loc}$ for 
the solution to the biharmonic obstacle problem with nonzero obstacle and zero boundary conditions.
In this section we show that the statement remains true in our setting, with a quantitative estimate of
$ \Arrowvert \Delta u \rVert_{L^\infty}$.

\begin{lemma} \label{l2}
 The solution to the biharmonic obstacle problem satisfies the following equation in the 
 distribution sense
 \begin{equation} \label{mu}
 \Delta^2u=\mu_u,
 \end{equation}
 where $\mu_u$ is a positive measure on $\Omega$.
\end{lemma}

\begin{proof}
For any nonnegative test function $ \eta \in C_0^{\infty}(\Omega)$, the function
$u+\varepsilon \eta$ is obviously admissible for any $\varepsilon >0$. Hence  
$J[u+\varepsilon \eta]\geq J[u]$, consequently
\begin{equation*}
\int{\varepsilon^2(\Delta \eta)^2+2\varepsilon \Delta u \Delta \eta} \geq 0,
\end{equation*}
and after dividing by $\varepsilon$ and letting $\varepsilon$
go to zero, we obtain
\begin{equation*}
\int{\Delta u \Delta \eta} \geq 0,
\end{equation*}
for all $ \eta \in C_0^{\infty}(\Omega)$, $\eta \geq 0$, so
 $ \Delta^2 u \geq 0$ in the sense of distributions.

\par
Let us consider the following linear functional defined on the space $ C_0^{\infty}(\Omega)$,
\begin{equation*}
\Lambda (\eta)=\int_\Omega \Delta u \Delta \eta.
\end{equation*}

\par
Then $\Lambda $ is a continuous linear functional on $C_0^\infty(\Omega)$, hence it is a distribution. 
According to the Riesz theorem, a positive 
distribution is a positive measure, let us denote this measure by  $\mu:= \mu_u$. 
Then $ \Delta^2 u = \mu_u $ in the sense that
\begin{equation*}
\int_\Omega \Delta u \Delta \eta = \int_\Omega \eta d\mu_u.
\end{equation*}
for every $ \eta \in C_0^{\infty}(\Omega)$.

\end{proof}

\begin{corollary} \label{cor26}
 There exists an upper semicontinuous function $\omega$ in $\Omega$, such that
$ \omega = \Delta u \textrm{ a.e. in } \Omega$.
\end{corollary}

\begin{proof}
 For any fixed $x_0 \in \Omega $, the function 
\begin{align*}
\omega_r(x_0):= \fint_{B_r(x_0)} \Delta u(x) dx
\end{align*}
is decreasing in $r>0$, since $ \Delta u$ is subharmonic by Lemma \ref{l2}. Define $ \omega(x):= \lim_{r\rightarrow 0} \omega_r(x) $, 
then $\omega $ is an upper semicontinuous function. On the other hand 
$ \omega_r(x) \rightarrow \Delta u (x) $ as $r \rightarrow 0$  a.e., hence $ \omega = \Delta u $ a.e. 
in  $\Omega$.
\end{proof}

\par
 The next lemma is a restatement of the
corresponding result in \cite{CF}, Theorem 2.2.

\begin{lemma}\label{wbdd}
Let $ \Omega \subset \mathbb{R}^n $ be a bounded open set with a smooth boundary, and let $u $ be a solution to the biharmonic 
obstacle problem with zero obstacle.
Denote by $ S$ the support of the measure $\mu_u = \Delta ^2 u$ in $\Omega$, then 
 \begin{equation} \label{supp}
  \omega(x_0) \geq 0,~~ \textrm{ for every }x_0 \in S  .
 \end{equation}

\end{lemma}

\begin{proof}
The detailed proof of Lemma \ref{wbdd} can be found in the original paper \cite{CF} and in the
book \cite{Fri82}(pp. 92-94), so we will provide only a sketch, 
showing the main ideas.

\par
Extend $u $ to a function in $ W^{2,2}_{loc}(\mathbb{R}^n)$, and denote by  $ u_\varepsilon$ the 
$\varepsilon$-mollifier of $u$. 
Let $x_0 \in \Omega $, assume that  there exists a ball $ B_r(x_0) $, such that 
$u_\varepsilon \geq \alpha >0$ in $ B_r(x_0) $.
Let $ \eta \in C_0^{\infty}(B_r(x_0))$, $ \eta \geq 0$ and $ \eta = 1 $ in $ B_{r/2}(x_0)$. Then 
for any $ \zeta \in C_0^\infty ( B_{r/2}(x_0))$ and $0 < t < \frac{\alpha}{2 \lVert\zeta \rVert_\infty} $
the function
\begin{equation*}
 v= \eta u_\varepsilon+(1-\eta)u \pm t \zeta
\end{equation*}
is nonnegative and it 
satisfies the same boundary conditions as $u$. Hence
\begin{align*}
\int (\Delta u)^2 \leq \int \left( \Delta( \eta u_\varepsilon+(1-\eta)u \pm t \zeta ) \right) ^2,
\end{align*}
after passing to the limit in the last inequality as $ \varepsilon \rightarrow 0$, we obtain
\begin{align*}
\int (\Delta u)^2 \leq \int  ( \Delta u \pm t \Delta \zeta ) ^2,
\end{align*}
Therefore
\begin{equation*}
 \int {\Delta u \Delta \zeta }=0 ,
\end{equation*}
 for all $ \zeta \in C_0^{\infty}(B_{r/2}(x_0))$, hence $\Delta^2 u=0$ in $B_{r/2}(x_0)$ and $ x_0 \notin S $.
It follows that if $ x_0 \in S$, then
there exists $ x_m \in \Omega $, $ x_m \rightarrow x_0$, and $\varepsilon_m \rightarrow 0$,
such that 
\begin{equation*}
 u_{\varepsilon_m}(x_m)  \rightarrow 0, \textrm{ as }  m\rightarrow \infty.
\end{equation*}
 Then by Green's formula,
\begin{align*}
 u_{\varepsilon_m}(x_m)= \fint_{\partial B_\rho(x_m)} u_{\varepsilon_m} d \mathcal{H}^{n-1}
 -\int_{B_\rho(x_m)} \Delta u_{\varepsilon_m}(y) V(x_m-y)dy,
\end{align*}
where $ \rho <  \dist (x_0, \partial \Omega)$ and $- V(z) $ is Green's function for Laplacian in the ball
$ B_\rho(0)$. Hence
\begin{align*}
 \liminf_{m\rightarrow\infty} \int_{B_\rho(x_m)} \Delta u_{\varepsilon_m}(y) V(x_m-y)dy \geq 0,
\end{align*}
 Then 
it follows from the convergence of the mollifiers and the upper semicontinuity of
$\omega$, that $\omega(x_0)\geq 0$, for any $ x_0\in S$.

\end{proof}

\par
Knowing that $ \Delta u $ is a subharmonic function, and $ \omega \geq 0$ on the support of $ \Delta^2 u$,
 we can 
show that $ \Delta u $ is locally bounded (Theorem 3.1 in \cite{CF}).

\begin{theorem} \label{l3}
 Let $ u $ be the solution to the biharmonic obstacle problem with zero obstacle in
 $ \Omega $, $ B_1 \subset \subset \Omega$. Then 
 \begin{equation} \label{inf}
  \Arrowvert \Delta u \rVert_{L^\infty(B_{1/3})} \leq C  \Arrowvert u \rVert_{W^{2,2}(\Omega)},
  \end{equation}
 where the constant $C>0$ depends on the space dimension $n $ and on $\dist(B_1, \partial\Omega)$.
\end{theorem}

\begin{proof}
The detailed proof of the theorem can be found in the original paper \cite{CF}, Theorem 3.1, and in the book \cite{Fri82},
pp. 94-97. Here we will only provide a sketch of the proof.

\par
Let $ \omega$ be the upper semicontinuous equivalent of $ \Delta u$ and $ x_0 \in B_{1/2}$, then 
\begin{equation*}
 \omega (x_0) \leq \fint_{B_{1/2}(x_0)} \Delta u(x) dx,
\end{equation*}
 since $ \omega $ is a subharmonic function. Applying H\"{o}lder's inequality, we obtain
 \begin{equation} \label{above}
  \omega(x_0) \leq \lvert B_{1/2}\rvert^{-\frac{1}{2}} \lVert \Delta u \rVert_{L^2(B_1)}.
 \end{equation}

\par
It remains to show that $ \Delta u $ is bounded from below in $B_{1/2}$. 
Let $ \zeta \in C_0^\infty(B_1)$, $ \zeta =1 $ in $ B_{2/3}$ and $ 0 \leq \zeta \leq 1$ elsewhere.
Referring to \cite{Fri82}, p.96, the following formula holds for any $ x \in B_{1/2}$
\begin{equation} \label{cf}
 \omega(x)= - \int_{B_{1/2}} V(x-y) d \mu- 
 \int_{B_1\setminus B_{1/2}} \zeta(y) V(x-y) \Delta^2 u dy + \delta(x),
\end{equation}
where $ V$ is Green's function for the unit ball $B_1$, and $ \delta$ is a bounded function,
\begin{equation} \label{del}
 \Arrowvert \delta \rVert_{L^\infty( B_{1/2})} \leq C_n
\Arrowvert \Delta u \rVert_{L^2(B_1)}.
\end{equation}
Denote
\begin{equation*}
 \tilde{V}(x):= \int_{B_{1/2}} V(x-y) d\mu (y),
\end{equation*}
then $ \tilde{V}$
is a superharmonic function in $\mathbb{R}^n$, and the measure $\upsilon:=\Delta \tilde{V} $ is supported on 
$S_0:= B_{1/2}\cap S$, moreover according to Lemma \ref{wbdd}, \eqref{supp}
\begin{equation*}
 \tilde{V} (x) \leq -\omega(x)+\delta(x)\leq \delta(x) ~\textrm{ on } S_0.
\end{equation*}
Taking into account that $ \tilde{V}(+\infty) <\infty$, the authors in \cite{CF} apply
Evans maximum principle, \cite{L} to the superharmonic function $ \tilde{V}-\tilde{V}(+\infty)$,
and conclude that 
\begin{equation}
 \tilde{V}(x)\leq  \Arrowvert \delta \rVert_{L^\infty( B_{1/2})}~ \textrm{ in }~\mathbb{R}^n.
\end{equation}

\par
It follows from equation \eqref{cf} that 
\begin{equation}\label{below}
 \omega(x) \geq - \Arrowvert \delta \rVert_{L^\infty( B_{1/2})}-c_n \mu_u(B_1)+\delta(x),
\end{equation}
for any $ x\in B_{1/3}$.

\par
Let $\eta \in C_0^\infty(\Omega)$ be a nonnegative function,  such that $ \eta=1$ in $B_1$ and 
$ 0 \leq \eta \leq1$ in $ \Omega$. Then
\begin{equation*}
 \mu_u(B_1)\leq \int_\Omega \eta d\mu_u =\int_\Omega \Delta u \Delta \eta  \leq 
 \Arrowvert \Delta u \rVert_{L^2(\Omega)} \Arrowvert \Delta \eta \rVert_{L^2(\Omega)},
\end{equation*}
and $ \eta $ can be chosen such that $  \Arrowvert \Delta \eta \rVert_{L^2(\Omega)} \leq C(\dist(B_1, \partial \Omega))$.
Hence
\begin{equation} \label{b1}
 \mu_u(B_1)\leq C \Arrowvert \Delta u \rVert_{L^2(\Omega)},
\end{equation}
where the constant $ C>0$ depends on the space dimension and on $ \dist(B_1, \partial \Omega)$.

\par
Combining the inequalities \eqref{above} and \eqref{below} together with \eqref{b1}, \eqref{del},
we obtain \eqref{inf}.
\end{proof}

\begin{corollary} \label{c1alpha}
 Let u be the solution to the biharmonic obstacle problem in $\Omega$. Then  $ u\in C^{1,\alpha}_{loc}$, 
 for any $0<\alpha<1$, and
 \begin{equation} \label{1alpha}
  \Arrowvert u \rVert_{C^{1,\alpha}(K)}\leq C\Arrowvert u \rVert_{W^{2,2}(\Omega)},
 \end{equation}
 where the constant $C$ depends on the space dimension and $\dist(K, \partial \Omega)$.
 \end{corollary}
 
 \begin{proof}
  It follows from Theorem \ref{l3} via a standard covering argument, that 
   \begin{equation*}
  \Arrowvert \Delta u \rVert_{L^\infty(K)}\leq C \Arrowvert u \rVert_{W^{2,2}(\Omega)}.
 \end{equation*}
 Then inequality \eqref{1alpha} follows from the Calder\'{o}n-Zygmund inequality and the Sobolev embedding theorem.
 \end{proof}

 According to Corollary \ref{c1alpha}, $ u $ is a continuous function in $\Omega$, and therefore
 $\Omega_u:=\{u>0\}$ is an open subset of $\Omega$. We define the free boundary
\begin{equation}
\Gamma_u = \partial\Omega_u \cap \Omega.
\end{equation}
It follows from our discussion that the measure $\mu_u= \Delta^2 u$ is supported on $\Gamma_u$.

\section{Regularity of the free boundary}

In this section we investigate the regularity of the free boundary $\Gamma_u$, under the assumption that 
the solution to the biharmonic obstacle problem is close to the one-dimensional
solution $\frac{1}{6}(x_n)^3_+$.

\subsection{One-dimensional solutions}

First we find the explicit solution to the  biharmonic obstacle problem in the interval $ (0,1) \subset \mathbb{R}$.

\begin{example} \label{e1}
The minimiser $ u_0$ of the functional
\begin{equation}
J[u]=\int_0^1 (u''(x))^ 2 dx,
\end{equation}
over nonnegative functions $u \in W^{2,2}(0,1)$,  with boundary conditions
$u(0)=1,u'(0)=\lambda<-3$ and $u(1)=0,u'(1)=0$, is a piecewise $3$-rd order polynomial,
\begin{equation}
 u_0(x)= \frac{\lambda^3}{3^3} \left(x+ \frac{3}{\lambda} \right)_-^3, ~x \in (0,1),
\end{equation}
hence $ u_0 \in C^{2,1}(0,1)$.
\end{example}

\begin{proof}
Let $ u_0 $ be the minimiser to the given biharmonic 
obstacle problem.  
If $ 0<x_0 <1$, and $ u_0(x_0)>0$, then $\int u_0'' \eta'' = 0$, for all
infinitely differentiable functions $\eta$ compactly supported in a small ball
centered at $x_0$. Hence the minimiser $ u_0$ has a fourth order derivative, $ u_0^{(4)}(x) =0 $ if 
$ x \in \{ u_0>0 \}$.
Therefore $u_0$ is a piecewise polynomial of degree less than or equal to three.
Denote by $\gamma \in (0,1]$ the first point where the
graph of $u_0$ hits the $x$-axes. Our aim is find the explicit value of $\gamma$. 
Then we can also compute the minimiser $ u_0$.

\par
Observe that $u_0(\gamma)=0$, and $u'_0(\gamma)=0$, since $u'_0$ is an
absolutely continuous function in $(0,1)$.
Taking into account the boundary conditions at the points $0$ and $\gamma$,
we can write $u_0(x)= ax^ 3+bx^2+\lambda x+1$ in $(0,\gamma)$,  where
\begin{equation*}
a=\frac{\lambda \gamma +2}{\gamma^3}, ~~ b=-\frac{2\lambda \gamma +3}{\gamma^2}.
\end{equation*}
We see that the point $\gamma$ is a zero of second order for the third order 
 polynomial $u_0$, and
$u_0 \geq0$ in $(0, \gamma]$. That means the third zero is not on the open interval 
$(0, \gamma)$, hence $\gamma \leq -\frac{3}{\lambda}$.

\par
Consider the function
\begin{equation*}
 F(\gamma):=\int_0^{\gamma}(u''(x))^ 2 dx,
\end{equation*}
then  
$ F(\gamma)=\frac{4}{\gamma^3}(\lambda^2\gamma^2+3\lambda \gamma+3)$. 
Hence $F'(\gamma)=-\frac{4}{\gamma^4}(\lambda \gamma +3)^2$, showing that 
the function $F$ is decreasing, so it achieves minimum at the point $\gamma = -\frac{3}{\lambda}$.
Therefore we may conclude that 
\begin{equation}
 u_0(x)= \frac{\lambda^3}{3^3} \left(x+ \frac{3}{\lambda} \right)_-^3, ~x \in (0,1),
\end{equation}
and $\gamma = -\frac{3}{\lambda} $ is a free boundary point.
Observe that $u''(\gamma)=0$, and $ u''$ is a continuous function, but $ u'''$ has a jump discontinuity
at the free boundary point $\gamma = -\frac{3}{\lambda}  $.

\end{proof}

\par
The example above characterises one-dimensional solutions. It also 
 tells us that one-dimensional solutions are $C^{2,1}$, and in general are not $C^3$.

\subsection{The class $ \mathscr{B}_\kappa^\varrho(\varepsilon)$ of solutions to the biharmonic obstacle problem }

Without loss of generality, we assume that $ 0\in \Gamma_u$, and study the regularity of the free
boundary, when $u \approx \frac{1}{6}(x_n)^3_+$.

\par
Let us start by recalling the definition of non-tangentially accessible domains, \cite{JK}.
\begin{definition} \label{nta}
 A bounded domain $ D \subset \mathbb{R}^n$ is called non-tangentially accessible (abbreviated NTA)
 when there exist constants $M$,  $ r_0 $ and a function $ l :\mathbb{R}_+ \mapsto \mathbb{N}$ such that 
 \begin{enumerate}
  \item 
  $ D$  satisfies the corkscrew condition; that is 
  for any $ x_0 \in \partial D$ and any $r<r_0$, there exists $P=P(r,x_0) \in D$ such that
  \begin{align}
   M^{-1}r< \lvert P-x_0\rvert<r ~ \textrm{ and } ~\dist(P, \partial D) > M^{-1}r.
  \end{align}
  \item
  $ D^c:= \mathbb{R}^n \setminus D$ satisfies the corkscrew condition.
\item
Harnack chain condition; if $ \epsilon>0 $ and $ P_1, P_2 \in D$, $ \dist (P_i,\partial D)>\epsilon$, 
  and $ \lvert P_1-P_2 \rvert<C \epsilon$, then there exists a Harnack chain from $P_1$ to $P_2$ whose length $ l$
  depends on $C $, but not on $ \epsilon$, $l=l(C)$.
  A Harnack chain from $P_1$ to $P_2$ is a chain of balls $ B_{r_k}(x^k)$, $k=1,...,l $ such that $ P_1 \in B_{r_1}(x^1)$,
   $ P_2 \in B_{r_l}(x^l)$, $ B_{r_k}(x^k) \cap  B_{r_{k+1}}(x^{k+1}) \neq \emptyset $, and
   \begin{align}
    M r_k> \dist( B_{r_k}(x^k), \partial D)>M^{-1} r_k.
   \end{align}

 \end{enumerate}

 \end{definition}

\par
Let us define rigorously, what we mean by $u \approx \frac{1}{6}(x_n)^3_+$.

\begin{definition} \label{d1}
Let $u \geq 0$ be the solution to the biharmonic obstacle problem in a domain $\Omega$, $ B_2 \subset\subset \Omega$ and
assume that $ 0 \in \Gamma_u$ is a free boundary point.
We say that 
$ u \in \mathscr{B}^{\varrho}_{\kappa}(\varepsilon)$, if the following assumptions are satisfied:
\begin{enumerate}
\item \label{itm:first}
$ u $ is almost one dimensional, that is
\begin{equation*}
 \Arrowvert \nabla' u\rVert_{W^{2,2}(B_2)} \leq \varepsilon,
  \end{equation*}
  where $ \nabla':= \nabla-e_n \frac{\partial}{\partial x_n}$.

\item \label{itm:second}
The set  $\Omega_u :=\{u > 0\} $ is an NTA domain with constants  
$  r_0=M^{-1}= \varrho$, and with a function $l$, indicating the length of a Harnack chain.

\item \label{itm:third}
There exists $2 >t>0$, such that $ u=0$ in $ B_2 \cap \{ x_n < -t \}$.

\item \label{itm:forth}
We have the following normalisation
\begin{equation} \label{nn}
 \lVert D^3 u \rVert_{L^2(B_1)}= \frac{1}{6} \left\| D^3  (x_n)_+^3 \right\|_{L^2(B_1)}
 =\frac{\lvert B_1\rvert^\frac{1}{2}}{2^\frac{1}{2}}:= \omega_n,
\end{equation}
and we also assume that 
\begin{equation} \label{bdd32}
 \lVert D^3 u \rVert_{L^2(B_2)} < \kappa, 
\end{equation}
where $ \kappa > \frac{1}{6} \left\| D^3  (x_n)_+^3  \right\|_{L^2(B_2)}= 2^\frac{n}{2}\omega_n$.

\end{enumerate}
\end{definition}

In the notation of the class $ \mathscr{B}_\kappa^\varrho(\varepsilon)$ we did not include the length
function $ l$, since 
later it does not appear in our estimates.  For the rest of this paper we will assume that we have a fixed 
length function $l$.
Later on in Corollary \ref{c2} we will see
that the precise value of the parameter $t $ in assumption \ref{itm:third} is not very important, and therefore
we also omit the parameter $t$ in our notation.

\par
Evidently $\frac{1}{6} (x_n)_+^3 \in \mathscr{B}^{\varrho}_{\kappa}(\varepsilon)$, for any $\varepsilon>0$
and $\varrho>0$. 
Next we show that if $u\in \mathscr{B}_\kappa^{\varrho}(\varepsilon)$, with $\varepsilon >0$ small,
then $ u \approx \frac{1}{6} (x_n)_+^3$ in $ W^{3,2}(B_1)$.

\par
From now on $\kappa >2^\frac{n}{2}\omega_n$ and $1>\varrho >0$ are fixed parameters.

\begin{lemma} \label{l4}
There exists a modulus of continuity $\sigma=\sigma(\varepsilon)\geq 0 $, such that 
\begin{equation} \label{near}
 \left\| u(x)-\frac{1}{6} (x_n)_+^3\right\| _{W^{3,2}(B_1)}\leq \sigma(\varepsilon),
\end{equation}
for any $u\in \mathscr{B}^{\varrho}_\kappa (\varepsilon)$.
\end{lemma}

\begin{proof}
We argue by contradiction. Assume that there exist $\sigma_0>0$ and a 
sequence of solutions, $ u^j \in \mathscr{B}_\kappa^{\varrho}({\varepsilon_j})$, such that
\begin{equation*} 
 \Arrowvert \nabla ' u^j  \rVert_{W^{2,2}(B_2)} = \varepsilon_j \rightarrow 0, 
 \end{equation*}
but
\begin{equation} \label{sigma0}
 \left\| u^j(x)-\frac{1}{6} (x_n)^3_+ \right\| _{W^{3,2}(B_1)} > \sigma_0 >0.
 \end{equation}

According to assumption \ref{itm:forth} in Definition \ref{d1},
$ \Arrowvert D^3 u^j \rVert_{L^2(B_2)} < \kappa $ and according to assumption \ref{itm:second} the functions
$u^j$ are vanishing
on an open subset of $B_2$. Therefore 
it follows from the Poincar\'{e} inequality that 
$ \Arrowvert u^j \rVert_{W^{3,2}(B_2)} \leq C(\varrho,  n)\kappa $.
Hence  
 up to a subsequence $ u^j \rightharpoonup u^0 $ weakly in $ W^{3,2}(B_2)$,
 $ u^j \rightarrow u^0$ strongly in $ W^{2,2}(B_2)$ and  according to
 Corollary \ref{c1alpha}
 $ u^j \rightarrow u^0 $ in $C^{1,\alpha}(B_{3/2})$. 
 Hence 
 \begin{align*}
 \Arrowvert \nabla' u^0 \rVert_{W^{1,2}(B_2)}= \lim_{j\rightarrow \infty}
 \Arrowvert \nabla' u^j \rVert_{W^{1,2}(B_2)}\leq\lim_{j\rightarrow \infty}\varepsilon_j =0.
 \end{align*}
 This implies that $ u^0$ is a
 1-dimensional solution (depending only on the variable $x_n$).
 Example \ref{e1} tells us that one-dimensional solutions in the interval $(-2,2)$ have the form 
 \begin{equation*}
  u^0(x_n) = c_1(x_n-a_1)^3_- + c_2(x_n -a_2)^3_+,
  \end{equation*}
  where $ c_1, c_2 \geq 0$ and 
 $-2 \leq a_1 \leq a_2 \leq 2 $ are constants. According to assumption \ref{itm:third} in Definition \ref{d1},
 $ u^0= c (x_n-a)^3_+$. In order to obtain a contradiction to assumption \eqref{sigma0}, we 
 need to show that $ u^j \rightarrow  u_0= \frac{1}{6}(x_n)_+^3$ in $ W^{3,2}(B_1)$. The proof of the last statement
 can be done in two steps.

 \par
\textit{ Step 1:} We show that 
\begin{equation}\label{32loc}
 u^j \rightarrow c(x_n-a)_+^3~\textrm{ in } W^{3,2}(B_1).
\end{equation}
Denote $ u^j_n:= \frac{\partial u^j}{\partial x_n} \in W^{2,2}(B_2) $, $ j \in \mathbb{N}_0$, and
let $ \zeta \in C_0^\infty(B_\frac{3}{2})$ be a nonnegative function, such that 
 $ \zeta \equiv 1 $ in $B_1$. According to Lemma \ref{zeta},
 \begin{equation*}
   0\geq  \int_{B_2} \Delta (\zeta u^j_n) \Delta u^j_n = 
\int_{B_2} u^j_n \Delta \zeta  \Delta u^j_n+ \int_{B_2}\zeta ( \Delta u^j_n )^2+
2 \int_{B_2} \nabla \zeta \nabla u^j_n \Delta u^j_n,
 \end{equation*}
and therefore
\begin{equation}\label{liminf1}
\begin{aligned}
 \limsup_{j\rightarrow \infty} \int_{B_2} \zeta (\Delta u^j_n )^2 \leq 
 - \lim_{j\rightarrow\infty} \int_{B_2} u^j_n \Delta \zeta  \Delta u^j_n-
 2 \lim_{j\rightarrow\infty} \int_{B_2} \nabla \zeta \nabla u^j_n \Delta u^j_n \\
 =
 -\int_{B_2}  u^0_n \Delta \zeta  \Delta u^0_n -
 2 \int_{B_2} \nabla \zeta \nabla u^0_n \Delta u^0_n =\int_{B_2} \zeta (\Delta u^0_n)^2,
\end{aligned}
\end{equation}
where in the last step we used integration by parts.

\par
On the other hand,
since $ \Delta u^j_n \rightharpoonup \Delta u^0_n$ weakly in $ L^2(B_2)$, it follows that
\begin{equation}\label{liminf2}
 \liminf_{j\rightarrow\infty} \int_{B_2} \zeta (\Delta u^j_n)^2 \geq 
  \int_{B_2} \zeta (\Delta u^0_n)^2.
\end{equation}

Therefore, we may conclude from \eqref{liminf1} and \eqref{liminf2} that
\begin{equation*}
\lim_{j\rightarrow\infty} \int_{B_2} \zeta (\Delta u^j_n)^2 = 
  \int_{B_2} \zeta (\Delta u^0_n)^2.
\end{equation*}
Hence we obtain
\begin{align*}
  \frac{\partial \Delta u^j}{\partial x_n} \rightarrow  \frac{\partial \Delta u^0}{\partial x_n} 
  \textrm{ in } ~ L^2(B_1).
  \end{align*}
Similarly $  \frac{\partial \Delta u^j}{\partial x_i} \rightarrow 0$ in 
$ L^2(B_1)$, for $i=1,..., n-1$.
Knowing that 
\begin{align*}
\Arrowvert \nabla \Delta u^j -\nabla \Delta u^0\rVert_{L^2(B_1)} \rightarrow 0, ~\textrm{ and }~
 \Arrowvert  u^j - u^0\rVert_{W^{2,2}(B_2)} \rightarrow 0,
 \end{align*}
we may apply the Calder\'{o}n-Zygmund inequality, and conclude  \eqref{32loc}.
 Recalling that $\Arrowvert D^3 u^j \rVert_{L^2(B_1)}=\omega_n$, we see that 
 \begin{equation} \label{u0omegan}
  \Arrowvert D^3 u^0 \rVert_{L^2(B_1)}=\omega_n.
 \end{equation}
Since $ u^0 = c (x_n-a)_+^3\geq 0$, it follows that 
\begin{equation*}
 \Arrowvert D^3 u^0 \rVert^2_{L^2(B_1)}= c^2 \mathcal{L}^n( B_1 \cap \{x_n > a\} ) > 0,
 \end{equation*}
 hence 
 \begin{equation} \label{*}
   c >0~ \textrm{ and } ~  a<1.
 \end{equation}

 \par
 \textit{Step 2:} We show that $ a=0$ and $ c=\frac{1}{6}$.
 Taking into account that $u^j\rightarrow u^0$ in $C^{1,\alpha}$ and
 $ u^j(0)=0$, we conclude that 
 $ u^0(0)= 0$, thus $ a\geq 0 $. 
 Assume that  $ a>0$. Since $ 0 \in \Gamma_j$, and $ \Omega_j$ is an NTA domain, 
 there exists 
 $ P_j=P(r,0) \in \Omega_j$, for $ 0<r<\min ( \varrho, a/2)$ as in the corkscrew conditon,
 \begin{equation*}
  \varrho r <\lvert P_j\rvert <r~ \textrm{and} ~\dist(P_j, \partial \Omega_j)> \varrho r.
 \end{equation*}
 Therefore up to a subsequence 
 $ P_j \rightarrow P_0 $, hence  $~r \varrho\leq \lvert P_0\rvert \leq r$, $~ B_{r'}(P_0) \subset \Omega_j, $ for all $ j$
 large enough, where $  0<r' <r \varrho$ is a fixed number.
 Since we have chosen $ r < a/2$, we may conclude that
 \begin{equation*}
 B_{r'}(P_0) \subset \{ x_n < a\} \cap \Omega_j .
 \end{equation*}
 Thus $ \Delta u^j $ is a sequence of harmonic functions in the ball $ B_{r'}(P_0)$, and therefore
  \begin{equation} \label{harn}
   \Delta u^j \rightarrow 0 \textrm{~~locally uniformly in} ~B_{r'}(P_0),
  \end{equation}
  according to \eqref{32loc}.
  
  \par
    Let $ Q:= e_n $, then $ u^0(Q)= c (1-a)^3 >0$, since $u^j \rightarrow u^0 $ uniformly in $ B_{ 3/2}$, we see that
    $ u^j(Q) >0$ for large $j$, and 
    $Q \in \Omega_j$. Therefore there exists a Harnack chain connecting $P_0$ with $Q$; 
 $\{  B_{r_1}(x^1),  B_{r_2}(x^2),...,   B_{r_l}(x^l) \} \subset \Omega_j$, whose length $l$ does not depend on $j$.
 Denote by $ K^j := \cup_i B_{r_i}(x^i) \subset \subset \Omega_j $, and let $ V^j \subset \subset K^j \subset \subset \Omega_j $
 where
  $V^j$ is a regular domain, such that $  \dist(K^j,\partial V^j )$ and $ \dist(V^j,\partial \Omega_j )$ depend 
  only on $ r$ and $\varrho$.

 \par
 Let $ w^j_+$ be a harmonic function in $ V^j$, with boundary conditions $ w^j_+= (\Delta u^j)_+ \geq 0 $ on $ \partial V^j$,
 then $ w^j_+ -\Delta u^j$ is a harmonic function in $V^j$, and $ w^j_+ -\Delta u^j = (\Delta u^j)_- \geq 0 $ on 
 $ \partial V^j$ , hence 
 \begin{equation}\label{wj+}
  0 \leq w^j_+-\Delta u^j \leq  \Arrowvert (\Delta u^j)_-\rVert_{L^\infty(V^j)}  \textrm{  in  }~ V^j.
 \end{equation}
Let us observe that $  \Delta u ^j \rightarrow \Delta u^0 =6c (x_n-a)_+$ implies that
$ \Arrowvert (\Delta u ^j )_-\rVert_{L^2(B_2)} \rightarrow 0$. Since $ (\Delta u^j)_-$ is a 
subharmonic function in $ \Omega_j$, and $ V^j \subset \subset \Omega_j$ it follows that
\begin{equation*}
 \Arrowvert (\Delta u^j)_-\rVert_{L^\infty(V^j)} \leq 
C(n,r, \varrho)  \Arrowvert (\Delta u ^j )_-\rVert_{L^2(B_2)} \rightarrow 0.
\end{equation*}

 \par
 So $ w^j_+ $ is a nonnegative harmonic function in $ V^j$, and by the Harnack inequality 
 \begin{align*}
  C_H \inf_{  B_{r_l}(x^l)}  w^j_+ \geq \sup_{  B_{r_l}(x^l)}  w^j_+ \geq 
  w^j_+(e_n) \geq
   \Delta u^j(e_n) \geq \frac{1}{2} \Delta u_0(e_n)= 3c(1-a),
 \end{align*}
 if $ j$ is large, where $C_H$ is the constant in Harnack's inequality, it depends on $\varrho$ and $r$ but not on $j$. 
 Denote $ C(a, c):= 3c(1-a)>0$ by \eqref{*}.
Applying the Harnack inequality again, we see that 
\begin{equation*}
  C_H \inf_{  B_{r_{l-1}}(x^{l-1})}  w^j_+ \geq \sup_{B_{r_{l-1}}(x^{l-1})  }  w^j_+ 
\geq  \inf_{  B_{r_l}(x^l)} w^j_+ > \frac{C(a, c)}{C_H}.
 \end{equation*}
 Inductively, we obtain that 
 \begin{equation} \label{ch}
   C_H \inf_{ B_{r_1}(x^1)}  w^j_+ \geq  \sup_{ B_{r_{1}}(x^{1})}  w^j_+ > \frac{C(a, c)}{C_H^{l-1}},
 \end{equation}
 where $ l$ does not depend on $j$. Hence 
 $ w^j_+(P_0) \geq \frac{C(a, c)}{C_H^{l}}$ for all $ j$ large, and according to \eqref{wj+},
 \begin{equation*}
  \lim_{j\rightarrow\infty}\Delta u^j (P_0) \geq \frac{C(a, c)}{C_H^{l}}>0,
  \end{equation*}
 the latter contradicts \eqref{harn}.
 Therefore we may conclude that  $a=0$.

\par
Recalling that $\Arrowvert D^3 u^0 \rVert_{L^2(B_1)}=\omega_n$, we see that $ c = \frac{1}{6}$, but then we obtain
  $  u^j \rightarrow \frac{1}{6}(x_n)_+^3 $ in $ W^{3,2}(B_1)$
which is a contradiction, since we assumed \eqref{sigma0}.

\end{proof}

\par
 Lemma \ref{l4} has an important corollary, which will be very useful in our later discussion.

 \begin{corollary} \label{c2}
 Let $ u $ be the solution to the biharmonic obstacle problem, $  u \in \mathscr{B}^\varrho_\kappa(\varepsilon)$.
 Then for any fixed $ t>0$ we have that $ u(x)=0 $ in $ B_2 \cap\{ x_n <-t \}$,  provided $\varepsilon=\varepsilon(t)>0$ is small.
 \end{corollary}
 
 \begin{proof}
  Once again we argue by contradiction. Assume that there exist $t_0 >0 $ and a sequence of solutions
  $u_j \in \mathscr{B}^\varrho_\kappa({\varepsilon_j}) $,  $ \varepsilon_j \rightarrow 0$, such that
  $ x^j \in B_2 \cap \Gamma_j $, and
   $ x^j_n < - t_0$. For $ 0<r<\min(\varrho, {t_0}/{2})$ choose $ P^j=P(r,x^j) \in \Omega_j$  as 
   in the corkscrew condition,
   \begin{align*}
    r \varrho < \lvert  x^j-P^j\rvert <r, ~~B_{r\varrho}(P^j) \subset \Omega_j.
   \end{align*}
 Upon passing to a subsequence, we may assume that $ P^j \rightarrow P^0$. Fix $0<r' <r \varrho$, then for large $j$
 \begin{equation*}
   B_{r'}(P^0) \subset \subset \Omega_j \cap \{ x_n < 0\}.
   \end{equation*}
   Hence $ \Delta u ^j$ is a sequence of harmonic functions in $ B_{r'}(P^0)$.   
 According to Lemma \ref{l4}, $ u^j \rightarrow \frac{1}{6}(x_n)_+^3$, and therefore $ \Delta u^j \rightarrow 0$ in 
 $ B_{r'}(P^0)$, and $ \Delta u^j(e_n)\rightarrow1$.
  Since $ \Omega_j$ is an NTA domain, there exists a Harnack chain connecting $P^0$ with $Q:= e_n \in \Omega_j$; 
 $\{  B_{r_1}(x^1),  B_{r_2}(x^2),...,   B_{r_k}(x^k) \} \subset \Omega_j$, whose length does not depend on $j$. 
 Arguing as in the proof of Lemma \ref{l4}, 
 we will obtain a contradiction to $ \Delta u_j \rightarrow 0$ in 
 $ B_{r'}(P^0)$.
 \end{proof}

\subsection{Linearisation}

Let $\{ u^j\} $ be a sequence of solutions in $ \Omega \supset\supset B_2$,  $u^j \in \mathscr{B}^{\varrho}_\kappa({ \varepsilon_j})$, and 
assume that  
$ \varepsilon_j \rightarrow 0 $ as $ j\rightarrow \infty$.
It follows from Lemma \ref{l4}, that up to a subsequence
\begin{equation} \label{uj}
 u^j \rightarrow \frac{1}{6}(x_n)_+^3 ~ \textrm{ in }~~ W^{2,2}(B_2) \cap C_{loc}^{1,\alpha}(B_2).
\end{equation}

\par
Let us denote
\begin{equation*} 
\delta^j_i:= \left\| \frac{\partial u^j}{\partial x_i} \right\| _{W^{2,2}(B_2)}. 
\end{equation*}
Without loss of generality we may assume that $ \delta^j_i> 0$, for all $ j  \in \mathbb{N}$. 
Indeed, if 
$ \delta^j_i =0$ for all  $j \geq J_0$ large, then
$ u^j$ does not depend on the variable $ x_i$, and the problem reduces to a lower dimensional
case.  Otherwise we may pass to a subsequence satisfying  $ \delta^j_i > 0 $ for all $j$.

\par
Denote 
\begin{equation} \label{vij}
 v^j_i: =\frac{1}{\delta^j_i}\frac{\partial u^j}{\partial x_i}, \textrm{ for } i=1,...,n-1,
\end{equation}
then $ \Arrowvert v^j_i \rVert_{W^{2,2}(B_2)}=1$.
Therefore up to a subsequence $ v^j_i $ converges to a function $v^0_i$ weakly in $ W^{2,2}(B_2) $ and
strongly in $ W^{1,2}(B_2)$.
For the further discussion we need strong convergence $ v^j_i \rightarrow v^0_i$ in $ W^{2,2} $, 
at least locally.

\begin{lemma} \label{l5}
 Assume that $ \{u^j\}$ is a sequence 
 of solutions in $ \Omega \supset \supset B_2$, 
 $ u^j \in \mathscr{B}^\varrho_\kappa ({\varepsilon_j})$, $\varepsilon_j\rightarrow 0$.
 Let $v^j_i$ be the sequence given
 by \eqref{vij}, and assume that
  $ v^j_i \rightharpoonup v^0_i$ weakly in $ W^{2,2}(B_2) $, strongly in $ W^{1,2}(B_2)  $, for $i=1,...,n-1 $, then
 \begin{equation} \label{v0}
  \Delta^2 v^0_i=0 \textrm{ in } B_2^+, ~~v^0_i \equiv 0 \textrm{ in } B_2 \setminus B_2^+.
 \end{equation}
Furthermore, for any $0<R<2$
\begin{equation}
 \Arrowvert v^j_i -v^0_i \rVert_{W^{2,2}(B_{R})} \rightarrow 0.
 \end{equation}
\end{lemma}

\begin{proof} 
Denote by 
 $ \Omega_j:= \Omega_{u^j},~\Gamma_j:= \Gamma_{u^j} $.
It follows from Corollary \ref{c2} that $ v^0_i \equiv 0 $ in $ B_2 \setminus B_2^+ $, 
hence $v^0_i=\lvert \nabla v^0_i \rvert =0 $ on $\{x_n=0\} \cap B_2$ in the trace sense. 
Moreover, if  $ K \subset \subset B_2^+$ is an open subset, then 
$ K \subset \Omega_j $ for large $j$ by \eqref{uj}. Hence $ \Delta^2 v^j_i =0$ in $ K$, and therefore 
$\Delta^2 v^0_i =0 $ in 
 $ B_2^+ $, and \eqref{v0} is proved.

\par
Now let us proceed to the proof of the strong convergence.
Let  $ \zeta \in C_0^\infty(B_{2})$ be a nonnegative function, such that  $  \zeta \equiv 1$ in 
$B_R$ and $0\leq \zeta \leq 1$ in $B_2$.
It follows from \eqref{v0} that
\begin{equation} \label{limvi0}
  0= \int_{B_{{2}}} \Delta v_i^0 \Delta (\zeta v^0_i) = \int_{B_{{2}}} v^0_i \Delta \zeta  \Delta v^0_i+ \int_{B_{{2}}}\zeta ( \Delta v^0_i )^2+
2 \int_{B_{{2}}} \nabla \zeta \nabla v^0_i \Delta v^0_i. 
\end{equation}

\par
According to Lemma \ref{zeta}
\begin{equation} \label{limvij}
  0 \geq \int_{B_{{2}}} \Delta (\zeta v^j_i) \Delta v^j_i = 
\int_{B_{{2}}} v^j_i \Delta \zeta  \Delta v^j_i+ \int_{B_{{2}}}\zeta ( \Delta v^j_i )^2+
2 \int_{B_{{2}}} \nabla \zeta \nabla v^j_i \Delta v^j_i,
\end{equation}
and therefore
\begin{equation*}
\begin{aligned}
 \limsup_{j \rightarrow \infty} \int_{B_{{2}}} \zeta (\Delta v^j_i )^2 \leq 
 -\lim_{j \rightarrow \infty}\int_{B_{{2}}} v^j_i \Delta \zeta  \Delta v^j_i 
 -2 \lim_{j \rightarrow \infty}\int_{B_{{2}}} \nabla \zeta \nabla v^j_i \Delta v^j_i \\
 = -\int_{B_{{2}}} v^0_i \Delta \zeta  \Delta v^0_i-
2 \int_{B_{{2}}} \nabla \zeta \nabla v^0_i \Delta v^0_i,
 \end{aligned}
\end{equation*}
where we used that $v^j_i \rightarrow v^0_i  \textrm{ in } W^{1,2}(B_2)$ and 
$\Delta v^j_i \rightharpoonup \Delta v^0_i \textrm{     in    } L^2(B_2)$.

\par
From the last inequality and \eqref{limvi0} we may conclude that 
\begin{equation}\label{lim1}
  \limsup_{j \rightarrow \infty} \int_{B_{{2}}} \zeta (\Delta v^j_i )^2 \leq 
  \int_{B_{{2}}} \zeta (\Delta v^0_i )^2.
\end{equation}

\par
On the other hand
\begin{equation} \label{lim2}
  \liminf_{j \rightarrow \infty} \int_{B_{{2}}} \zeta (\Delta v^j_i )^2 \geq 
  \int_{B_{{2}}} \zeta (\Delta v^0_i )^2
\end{equation}
follows from the weak convergence $ \Delta v^j_i \rightharpoonup \Delta v^0_i$ in $ L^2(B_2)$,
and we 
 may conclude from \eqref{lim1} and \eqref{lim2} that
\begin{equation*}
  \lim_{j \rightarrow \infty} \int_{B_{{2}}} \zeta (\Delta v^j_i )^2 =
  \int_{B_{{2}}} \zeta (\Delta v^0_i )^2.
\end{equation*}
Hence we obtain
$ \Arrowvert \Delta v^j_i - \Delta v^0_i \rVert_{L^{2}(B_R)} \rightarrow 0 $, and 
therefore
$ v^j_i \rightarrow v^0_i $ in $W^{2,2}_{loc}(B_2)$ according to the Calder\'{o}n-Zygmund inequality.
\end{proof}

\subsection{Properties of solutions in a normalised coordinate system}

Let us define
 \begin{equation} \label{resc}
  u_{r,x_0}(x):=\frac{u(rx+x_0)}{r^3}, \textrm{ for } x_0 \in \Gamma_u, ~x \in B_2, ~  r \in (0,1),
 \end{equation}
and $ u_r := u_{r, 0}$.
We would like to know how fast $\Arrowvert \nabla'  u_r \rVert_{W^{2,2}(B_2)}$
 decays with respect to  $\Arrowvert \nabla' u \rVert_{W^{2,2}(B_2)}$, for $r<1$. 
 In particular, the inequality
 \begin{equation} \label{good}
 \Arrowvert \nabla'  u_s \rVert_{W^{2,2}(B_2)} \leq \tau  \Arrowvert \nabla'  u\rVert_{W^{2,2}(B_2)} ,
 \end{equation}
for some $0<s,\tau<1$ would provide good decay estimates for 
$ \Arrowvert \nabla' u_{s^ k }\rVert_{W^{2,2}(B_2)}$, $k\in \mathbb{N}$. 

\par
We show that the inequality \eqref{good} holds in a special coordinate system depending on the solution $u$ and parameter $s>0$. Then  iterating the inequality \eqref{good} and  the coordinate system we obtain  the existence of the unit normal vector to the free boundary at the origin.

\par
Let us observe that 
 $\frac{1}{6}(\eta \cdot x)_+^3 \in \mathscr{B}_\kappa^\varrho(\varepsilon)$ if
$ \lvert \eta - e_n \rvert \leq C_n \varepsilon$, for some dimensional constant $C_n$.

\begin{definition} \label{d2}
 Let $u$ be the solution to the biharmonic obstacle problem. We say that the coordinate system is
 normalised with respect to $u $, if 
 \begin{align*}
  \inf_{\eta \in \mathbb{ R }^n,\lvert \eta \rvert=1 }
  \left\| \nabla'_\eta   \left( u(x)-\frac{1}{6}(\eta \cdot x)_{+}^3 \right) \right\| _{L^{2}(B_{2})} \\
  = \left\| \nabla'_{e_n}  \left( u(x)- \frac{1}{6}( x_n)_{+}^3 \right)\right\| _{L^{2}(B_{2})},
 \end{align*}
where $\nabla'_\eta:= \nabla - (\eta \cdot \nabla)\eta $, and $\nabla':=\nabla'_{e_n}$.
\end{definition}

\par
A minimiser $\eta $ always exists for a function
$ u\in \mathscr{B}_\kappa^\varrho(\varepsilon)$, and
since $ \nabla'_{-\eta}= \nabla'_{\eta}$, $-\eta $ is also a minimiser, thus
we always choose a minimiser satisfying the condition $ e_n\cdot \eta \geq 0$.
A normalised coordinate system always exists by choosing $\eta = e_n $ in the new coordinate system.

\begin{lemma} \label{integralDeltavi0}
Let $u$ be the solution to the biharmonic obstacle problem in  a normalised coordinate system with respect to $u $.
Then
\begin{equation} \label{integralDeltauij}
\int_{B_2} \frac{\partial  u}{\partial x_i} \frac{\partial u}{\partial x_n}dx=0, \textrm{ for all } 1 \leq i\leq n-1.
\end{equation}
 \end{lemma}

\begin{proof}
Let us observe that for every $\eta \in \mathbb{R}^n$,
\begin{equation*}
 \nabla'_\eta   \left( u(x)-\frac{1}{6}(\eta \cdot x)_{+}^3 \right) = \nabla'_\eta   u(x)
 \end{equation*}
 and 
 \begin{equation*}
 \left\| \nabla'_\eta   u\right\| _{L^{2}(B_{2})}^2 =
\left\| \nabla  u\right\| _{L^{2}(B_{2})}^2 - 
\left\| \eta \cdot \nabla   u\right\| _{L^{2}(B_{2})}^2 .
\end{equation*}

For any fixed $ 1 \leq i\leq n-1$,  and real number $-1<t <1$, let $\eta(t):= t e_i +\sqrt{1-t^2} e_n$.  By the definition of a normalised coordinate system, 
the function $ \varphi(t):=\left\| \eta \cdot \nabla  u\right\| _{L^{2}(B_{2})}^2  $, $ t\in (-1,1)$ has a local maximum at the point 
$ t=0$. Hence 
\begin{equation}
\varphi^\prime (0)=2\int_{B_2} \frac{\partial u}{\partial x_i} \frac{\partial u}{\partial x_n}dx=0,
\end{equation}
which implies \eqref{integralDeltauij}.
\end{proof}

\begin{lemma} \label{l6}
Assume that $u\in \mathscr{B}_\kappa^\varrho(\varepsilon) $ solves the biharmonic obstacle problem in a fixed coordinate system with basis vectors 
$\{ e_1,..., e_n\}$.
Let $\{ e_1^1,...,e^1_n \}$ be a normalised coordinate system with respect to $u$, and assume that 
$ e_n^1 \cdot e_n \geq 0 $. Then
\begin{equation*}
 \lvert e_n-e_n^1\rvert \leq C(n) \Arrowvert \nabla' u \rVert_{L^2(B_2)}\leq C(n) \varepsilon,
\end{equation*}
 if $\varepsilon$ is small, where $C(n) >0$ is a dimensional constant. 
\end{lemma}

\begin{proof}

According to Definition \ref{d2},
\begin{equation} \label{en1}
 \Arrowvert \nabla'_{e_n^1}   u \rVert_{L^{2}(B_2)}=
 \Arrowvert \nabla  u -e_n^1 ( e_n^1 \cdot \nabla  u)\rVert_{L^2(B_2)}\leq 
 \Arrowvert \nabla'  u \rVert_{L^2(B_2)} .
 \end{equation}
 It follows from the triangle inequality that
 \begin{equation} \label{big}
\begin{aligned}
 \left\| \frac{\partial u}{\partial x_n}-
 (e_n \cdot e_n^1)^2 \frac{\partial  u}{\partial x_n} \right\| _{L^2(B_2)} \leq 
 \left\| \frac{\partial u}{\partial x_n}-
 (e_n \cdot e_n^1) ( e_n^1 \cdot \nabla  u) \right\| _{L^2(B_2)} \\
 +\left\| (e_n \cdot e_n^1 ) ( e_n^1 \cdot \nabla  u) - 
 (e_n \cdot e_n^1 )^2
 \frac{\partial u}{\partial x_n} \right\| _{L^2(B_2)} \\
 \leq \Arrowvert \nabla_{e_n^1}' u \rVert_{L^2(B_2)}+  ( e_n\cdot e_n^1 )
 \Arrowvert e_n^1 \cdot \nabla' u \rVert_{L^2(B_2)} \leq 
2 \Arrowvert \nabla' u \rVert_{L^2(B_2)},
 \end{aligned}
 \end{equation}
 according to \eqref{en1}, and taking into account that $ 0 \leq   e_n \cdot e_n^1  \leq 1$.
 
  \par
Note that Lemma \ref{l4} implies that
$ \left\| \frac{\partial  u }{\partial x_n} \right\| _{L^2(B_2)} \approx \left\| \frac{x_n^2 }{2} \right\| _{L^2(B_2^+)} $ is uniformly bounded
from below by a dimensional constant if  $ \varepsilon>0$ is small.  We may conclude from \eqref{big} that 
\begin{equation*}
  1-(e_n \cdot e_n^1)^2 \leq 
  \frac{2 \Arrowvert \nabla'  u \rVert_{L^2(B_2)} }{
  \left\| \frac{\partial  u }{\partial x_n} \right\|_{L^2(B_2)} } 
  \leq C(n)  \Arrowvert \nabla'  u \rVert_{L^2(B_2)}.
\end{equation*}
Since $ 0 \leq e_n \cdot e_n^1 \leq 1 $, we get 
\begin{equation} \label{enen1}
 0\leq 1-e_n \cdot e_n^1 \leq 1-(e_n \cdot e_n^1)^2 \leq C(n)  \Arrowvert \nabla' u \rVert_{L^2(B_2)}.
\end{equation}

 \par
 Denote by $ (e_n^1)':= e_n^1- e_n (e_n \cdot e_n^1)$.  
It follows from the triangle inequality and \eqref{en1} that
\begin{equation*}
\begin{aligned}
 \Arrowvert (e_n^1)' (e_n^1 \cdot \nabla  u) \rVert_{L^2(B_2)} \leq
  \Arrowvert \nabla'  u -(e_n^1)' (e_n^1 \cdot \nabla u) \rVert_{L^2(B_2)} \\
 + \Arrowvert \nabla'  u \rVert_{L^2(B_2)} \leq 
 \Arrowvert \nabla'_{e_n^1}  u \rVert_{L^2(B_2)}+\Arrowvert \nabla'  u \rVert_{L^2(B_2)} \leq
 2  \Arrowvert \nabla'  u \rVert_{L^2(B_2)} .
  \end{aligned}
\end{equation*}
Hence 
\begin{equation*}
 \lvert (e_n^1)'\rvert \leq \frac{2 \Arrowvert \nabla'  u \rVert_{L^2(B_2)} }{
  \left\| e_n^1 \cdot \nabla  u  \right\|_{L^2(B_2)} }.
\end{equation*}
Let us choose  $ \varepsilon> 0$ small, then  $  \left\| e_n^1 \cdot \nabla  u  \right\|_{L^2(B_2)}$ is bounded from below 
by a dimensional constant according to
Lemma \ref{l4} and inequality \eqref{enen1}. Therefore we obtain
\begin{equation} \label{enn}
 \lvert (e_n^1)'\rvert 
  \leq C(n)  \Arrowvert \nabla'  u \rVert_{L^2(B_2)}.
\end{equation}

 \par
 Note that 
 \begin{equation*}
 \lvert e_n-e_n^1\rvert \leq  \lvert 1-e_n\cdot e_n^1 \rvert + \lvert (e_n^1)'\rvert.
\end{equation*}
Applying inequalities \eqref{enen1} and \eqref{enn} we obtain the desired inequality,
\begin{equation*}
 \lvert e_n-e_n^1\rvert  \leq 
 C(n) \Arrowvert \nabla'  u \rVert_{L^2(B_2)} \leq C(n)  \varepsilon,
\end{equation*}
and the proof of the lemma is now complete.

\end{proof}

\par
Lemma \ref{l6} provides an essential estimate, which will be useful in our later discussion.
Next we state another supporting lemma,  the proof of which is quite standard, but we include it for our convenience.
\begin{lemma} \label{estimateDeltav0}
\begin{enumerate}
  \item Let $ v $ be a biharmonic function in the ball $B_2$, then
\begin{equation}\label{esimateDeltavi0}
 \Arrowvert \Delta v \rVert_{L^2(B_1)} \leq C_n \Arrowvert  v \rVert_{L^2(B_{2})}.
\end{equation}
  \item If  $ v $ is a biharmonic function in the half-ball $B_2^+$, such that $v=|\nabla v|=0$ on $ \{ x_n=0\} \cap B_2$, then
\begin{equation}\label{esimateDeltavi0halfball}
 \Arrowvert \Delta v \rVert_{L^2(B_1^+)} \leq C_n \Arrowvert  v \rVert_{L^2(B^+_{2})}.
\end{equation}
\end{enumerate}
\end{lemma}
\begin{proof}
 Throughout   $ \zeta \in C_0^\infty(B_{2})$ is a fixed function, such that  $  \zeta \equiv 1$ in 
$B_1$, $ 0 \leq \zeta \leq 1$ in $B_{2}$. 

\par
1. If  $v$ is a biharmonic function in $B_2$, then 
\begin{align*} 
  0= \int_{B_{{2}}} \Delta v \Delta (\zeta^4 v) = \int_{B_{{2}}} 2 v(  \Delta \zeta^2+4|\nabla \zeta|^2)\zeta^2 \Delta v \\
  + \int_{B_{{2}}}\zeta^4 ( \Delta v )^2+
8 \int_{B_{{2}}} \zeta \nabla \zeta \nabla v \zeta^2 \Delta v.
\end{align*}
Hence by Cauchy's inequality,
\begin{equation} \label{3estimate1}
\begin{aligned}
 \int_{B_{{2}}}\zeta^4 ( \Delta v )^2=-  \int_{B_{{2}}} 2 v(  \Delta \zeta^2+4|\nabla \zeta|^2)\zeta^2 \Delta v
 -  8 \int_{B_{{2}}} \zeta \nabla \zeta \nabla v \zeta^2 \Delta v \\
 \leq 
 \frac{1}{4}\int_{B_{{2}}} \zeta^4 ( \Delta v)^2 +4 \int_{B_{{2}}} v^2 (\Delta \zeta^2+4|\nabla \zeta|^2 )^2
 + \frac{1}{4}\int_{B_{{2}}}\zeta^4 (\Delta v )^2 \\
 + 64 \int_{B_{{2}}} (\zeta \nabla \zeta \nabla v)^2
 \leq \frac{1}{2}\int_{B_{{2}}} \zeta^4 ( \Delta v)^2 + C_n  \Arrowvert  v \rVert^2_{L^2(B_{2})}  +
 C_n   \Arrowvert \zeta \nabla v \rVert^2_{L^2(B_{2})}  .
  \end{aligned}
  \end{equation}
On the other hand, 
\begin{equation*}
\begin{aligned}
  \int_{B_{2}} \zeta^2 | \nabla v |^2 =  \int_{B_{2}} \nabla (\zeta^2 v) \nabla v  -2 \int_{B_{2}} \zeta v \nabla \zeta \nabla v=\\
  - \int_{B_{2}} \zeta^2 v  \Delta v -2 \int_{B_{2}} \zeta v \nabla \zeta \nabla v   \leq
   \frac{1}{8C_n}\int_{B_2}\zeta^4 (\Delta v )^2+ 2C_n \int_{B_{{2}}} v ^2 \\
   +\frac{1}{2}  \int_{B_{2}} \zeta^2 | \nabla v |^2+2   \int_{B_{2}} | \nabla \zeta|^2  v ^2,
         \end{aligned}
    \end{equation*}
and therefore
\begin{equation}\label{3estimate2}
 C_n   \Arrowvert \zeta \nabla v \rVert^2_{L^2(B_2)} \leq \frac{1}{4} \int_{B_2}\zeta^4 ( \Delta v )^2+ \tilde{C}_n \int_{B_2} v ^2
  \end{equation}
Combining estimates \eqref{3estimate1} and \eqref{3estimate2}, we obtain
\begin{align*}
 \int_{B_2} \zeta^4 ( \Delta v )^2 \leq  \frac{3}{4} \int_{B_2} \zeta^4 ( \Delta v )^2 + \bar{C}_n \Arrowvert  v \rVert^2_{L^2(B_2)},
  \end{align*}
which implies \eqref{esimateDeltavi0}.

\par 
2. In order to prove the second part of the lemma, it is enough to observe that
\begin{equation*}
  \int_{B_2^+} \Delta v \Delta (\zeta^4 v) =0,
\end{equation*}
since $ \zeta^4 v\in W^{2,2}_0(B_2^+)$. The rest of the proof follows as in the first part.
\end{proof}

\begin{proposition} \label{p3}
For any  small number  $0 < s<2^{-n-4}e^{-1}n^{-2}$,  there exists $\varepsilon_0=\varepsilon_0(s)>0 $ small, such that if $ \varepsilon<\varepsilon_0$, then 
for any $ u \in \mathscr{B}_\kappa^\varrho(\varepsilon)$
 \begin{equation} \label{boundedrescalings}
  \left\| \nabla'  u_{2s} \right\| _{L^{2}(B_2)} \leq 
  C_n \left\| \nabla' u \right\|_{W^{2,2}(B_2)},
 \end{equation}
 where $C_n$ is a dimensional constant, not depending on $s$.
Furthermore, if the coordinate system is normalised  with respect to 
$u_{2s}$, then 
 \begin{equation} \label{tau}
  \left\| \nabla'  u_s \right\| _{W^{2,2}(B_2)} \leq 
  \tau  \left\| \nabla' u \right\|_{W^{2,2}(B_2)},
 \end{equation}
where $1>\tau >\Lambda_n s$ is a fixed number and $\Lambda_n$ is a dimensional constant to be specified.
\end{proposition}

\begin{proof} 
The proof of inequalities \eqref{boundedrescalings} and \eqref{tau} follows the exact same procedure, so we will mainly focus on the 
proof of the second one, since it is the core of the linearisation argument.

\par 
We want to show that  the inequality \eqref{tau} holds in a normalised coordinate system with respect to $u_{2s}$.
By the Cauchy-Schwarz inequality, it is enough to show that the inequality
\begin{equation*}
  \left\| \frac{ \partial  u_s }{\partial x_i}\right\| _{W^{2,2}(B_2)} \leq 
\tau  \left\|  \frac{ \partial u}{\partial x_i} \right\|_{W^{2,2}(B_2)}
 \end{equation*}
holds for any $i\in \{1,...,n-1\}$, provided $\varepsilon$ is small enough.
We argue by contradiction.  Assume that  there exists small numbers $0<s<2^{-n-4}e^{-1}n^{-2}$, $ \Lambda_n s<\tau <1$
and a sequence of solutions 
$ \{ u^j \} \subset \mathscr{B}_\kappa^\varrho(\varepsilon_j)$,  in a
 coordinate system normalised with respect to $u^j_{2s}$, such that $ \varepsilon_j \rightarrow 0$, as $j\rightarrow\infty$, but
 for some $i \in \{ 1,2,...,n-1 \}$
\begin{equation} \label{conts}
  \left\| \frac{ \partial  u^j_{s} }{\partial x_i}\right\| _{W^{2,2}(B_2)} >
\tau \left\|  \frac{\partial  u^j}{\partial x_i} \right\|_{W^{2,2}(B_2)}.
 \end{equation}

\par
Let $v^j_i$ be given by \eqref{vij}, then according to Lemma \ref{l5}, $ v^j_i\rightarrow v^0_i $ in 
$W^{2,2}_{loc}(B_2)$, where $v^0_i$ is a 
  biharmonic function in the half-ball $ \{ x_n >  0\} \cap B_2 $, satisfying $v^0_i=|\nabla v^0_i|=0$ on $\{x_n=0\}\cap B_2$. Inequality \eqref{conts} implies that 
  \begin{equation} \label{contsv}
  \left\| \frac{v_i^0(s\cdot)}{s^2}\right\| _{W^{2,2}(B_2)} \geq \tau.
 \end{equation}
  Lemma \ref{estimateDeltav0}, part 2.  and the Calder\'on-Zygmund inequality imply that 
  \begin{equation*}
\left\| \frac{v_i^0(2s\cdot)}{4s^2}\right\| _{W^{2,2}(B_1)} \leq C_n \left\| \frac{v_i^0(2s\cdot)}{4s^2}\right\| _{L^{2}(B_2)},
\end{equation*}
 hence
  \begin{equation} \label{contsv22}
 \left\| \frac{v_i^0(2s\cdot)}{4s^2}\right\| _{L^{2}(B_2)} \geq \frac{1}{C_n} \left\| \frac{v_i^0(2s\cdot)}{4s^2}\right\| _{W^{2,2}(B_1)} \geq 
 C_n  \left\| \frac{v_i^0(s\cdot)}{s^2}\right\| _{W^{2,2}(B_2)}\geq C_n \tau,
 \end{equation}
where $ C_n$ represents a general dimensional constant, and it does not depend neither on the function $v_i^0$ nor on the parameter $s$. 
We will derive a contradiction to \eqref{contsv}, if we show that $  \left\| \frac{v_i^0(2s\cdot)}{4s^2}\right\| _{L^{2}(B_2)}$ can be made arbitrarily small by choosing $ s>0$ small
initially.

\par
 Since $v^0_i$ is a 
  biharmonic function in the half-ball $ \{ x_n >  0\} \cap B_2 $ and $v^0_i=|\nabla v^0_i|= 0$
on $  \{ x_n = 0\} \cap B_2$,
we can apply the reflection principle for biharmonic functions, and extend $ v^0_i$ to a biharmonic function in the ball $B_2$, see for instance
\cite{Duffin} or \cite{Huber}. Let $\bar{v}^0_i$ denote the extended function given by Duffin's formula
\begin{equation} \label{Duffin}
\bar{v}^0_i(x',-x_n)=-v^0_i(x',x_n)+2x_n\frac{\partial v^0_i}{\partial{x_n}}(x',x_n)-x_n^2\Delta v^0_i(x',x_n), ~x_n >0.
\end{equation}
The formula \eqref{Duffin} implies that 
\begin{equation}\label{bddextension}
 \Arrowvert  \bar{v}^0_i \rVert_{L^2(B_{2})} \leq c_n \Arrowvert v^0_i \rVert_{W^{2,2}(B_{2})},
 \end{equation}
 where $c_n>0$ is yet another dimensional constant.

 \par
The function $ \bar{v}^0_i$ is biharmonic in the ball $B_2$, therefore analytic and it may be written as a Taylor series
 \begin{equation} \label{Taylor}
 \bar{v}^0_i(x)=  \sum_{|\alpha|=0}^\infty \frac{D^\alpha \bar{v}^0_i(0) }{\alpha!} x^\alpha=\sum_{k=0}^\infty b_k(x), \end{equation}
 where $ \alpha$ is a multiindex, and $ b_k$ is a homogeneous degree $k$ biharmonic 
polynomial. It follows from boundary conditions for the function $v^0_i$ on $ \{x_n=0\}$ that 
\begin{equation}
b_0=b_1\equiv 0, \textrm{ and }b_2(x)= \frac{\partial^2  \bar{v}^0_i(0)}{\partial x_n^2} \frac{x_n^2}{2}.
\end{equation}

\par
Lemma \ref{l4} implies that 
$ \frac{\partial  u^j}{\partial x_n} \rightarrow \frac{1}{2}(x_n^+)^2 $ in ${L^2(B_2)}$, and
according to Lemma \ref{l5},
$    \frac{v^j_i (2sx)}{4s^2} \rightarrow   \frac{v^0_i(2sx)}{4s^2}$ in ${W^{2,2}(B_2)}$ as
$j\rightarrow \infty$, and $ v^0_i= 0$ in $ B_2 \setminus B_2^+$.
By Lemma \ref{integralDeltavi0},
\begin{equation*}
\int_{B_2} v^j_i (2sx)\frac{\partial  u^j}{\partial x_n} (2sx)dx=\frac{1}{\delta^j_i}\int_{B_2} \frac{\partial  u^j} {\partial x_i }(2sx)\frac{\partial  u^j}{\partial x_n}(2sx)dx =0,
\end{equation*}
and after passing to the limit as $ j\rightarrow \infty$, we obtain that 
\begin{equation} \label{ai0ai1}
\int_{B_2^+}  v^0_i (2sx)x_n^2dx =0.
\end{equation}
Note that \eqref{ai0ai1} implies that 
\begin{equation*} 
 \left\|  \frac{v^0_i(2s\cdot)}{4s^2}-b_2 \right\|^2_{L^2(B_2^+)}
=  \left\| \frac{v^0_i(2s\cdot)}{4s^2} \right\|^2_{L^2(B_2^+)} +   \left\| b_2  \right\|^2_{L^2(B_2^+)} ,
\end{equation*}
hence 
\begin{equation}\label{nice}
 \left\|  \frac{v^0_i(2s\cdot)}{4s^2}  \right\|_{L^2(B_2^+)} \leq  \left\|  \frac{v^0_i(2s\cdot)}{4s^2}-b_2 \right\|_{L^2(B_2^+)}.
  \end{equation}
Next we show that $ \left\| v^0_i(2s\cdot)-4s^2b_2  \right\|^2_{L^2(B_2^+)}$ is of order $s^3$. By the triangle inequality,
\begin{equation}\label{longestimate}
\begin{aligned}
\left\|  \frac{v^0_i(2s\cdot)}{4s^2}-b_2 \right\|_{L^2(B_2^+)}=\left\| \sum_{k=3}^\infty (2s)^{-2} b_k (2s\cdot)  \right\|_{L^2(B_2^+)} \\
\leq \sum_{k=3}^\infty \left\| (2s)^{-2} b_k (2s\cdot)  \right\|_{L^2(B_2^+)}
=\sum_{k=3}^\infty (2s)^{k-2}  \left\| b_k \right\|_{L^2(B_2^+)}.
\end{aligned}
\end{equation}
Now it is time to refer to the estimates on derivatives for biharmonic functions (see Appendix A),
\begin{equation}
\begin{aligned}
b_k(x)=\sum_{|\alpha|=k}\frac{ D^\alpha  \bar{v}^i_0(0)}{\alpha!} x^\alpha, \textrm{ and }\\
|D^\alpha  \bar{v}^i_0(0) |\leq \frac{(2^{n+1}nk)^k}{r^{n+k}} 
\left( \left\|  \bar{v}^0_i  \right\|_{L^1(B_{r})}+\frac{r^2}{2(n+2)}
\left\| \Delta \bar{v}^0_i  \right\|_{L^1(B_{r})} \right).
\end{aligned}
\end{equation}
Hence
\begin{align*}
 \left\| b_k  \right\|_{L^2(B_2^+)} \leq \sum_{|\alpha|=k}\frac{(2^{n+1}nk)^k 2^{n+k}  }{|B_1|^\frac{1}{2}\alpha! } 
 \left( \left\| \bar{v}^0_i  \right\|_{L^1(B_{1})}  +\frac{1}{2(n+2)} \left\| \Delta \bar{v}^0_i  \right\|_{L^1(B_{1})}    \right) \\
\overset{\eqref{esimateDeltavi0}}{\leq} C_n {(2^{n+2}nk)^k } \sum_{|\alpha|=k}\frac{1}{\alpha! }  \left\|  \bar{v}^0_i  \right\|_{L^2(B_{2})} =
C_n \frac{(2^{n+2}nk)^k n^k }{ k!} \left\|  \bar{v}^0_i  \right\|_{L^2(B_{2})} \\
\leq  C_n{2^{k(n+2)}n^{2k} e^k } \left\| \bar{v}^0_i  \right\|_{L^2(B_{2})}  ,
   \end{align*}
 where we used Stirling's inequality in the last step.
 
 \par
Let  $ \lambda:= 2^{n+2} e n^{2}$ be  a fixed number, then by \eqref{longestimate},
\begin{equation} \label{As}
\begin{aligned}
\left\| \frac{v^0_i(2s\cdot)}{4s^2}-b_2 \right\|_{L^2(B_2^+)} \leq 
C_n s^{-2}\sum_{k=3}^\infty (2s)^{k} \lambda^{k} {\left\|  \bar{v}^0_i  \right\|_{L^2(B_{2})}} \\
=  C_n \frac{\lambda^3 s}{1-2\lambda s}\left\| \bar{v}^0_i  \right\|_{L^2(B_{2})}
\leq \tilde{C}_n s\left\| \bar{v}^0_i  \right\|_{L^2(B_{2})},
\end{aligned}
\end{equation}
where by assumption $2s \lambda<1/2$.

\par
Finally, combining the inequalities  \eqref{bddextension} ,  \eqref{nice} and \eqref{As}, we obtain
\begin{equation*}
 \left\| \frac{v^0_i(2s\cdot)}{4s^2}\right\| _{L^2(B_2)} \leq  c_n  \tilde{C}_n s \lVert v^0_i\rVert_{W^{2,2}(B_{2})} \leq A_n s,
\end{equation*}
where $A_n >0$ is a dimensional constant, and $s$ is fixed small number, $sA_n<1$. Let $\Lambda_n := A_n /C_n$, where $C_n $
is the dimensional constant in \eqref{contsv22}. Recalling that  $1>\tau >s\Lambda_n$, we derive a contradiction to \eqref{contsv22}.

\par 
The proof of the inequality \eqref{boundedrescalings} is very similar. Any biharmonic function $v$ in the half ball $B_2^+$, satisfying the boundary conditions $ v=|\nabla v|=0$ on $ \{x_n=0\}$ can be written as
\eqref{Taylor}.
Employing the estimates of derivatives of biharmonic functions, we can show that $  \left\| \frac{v(s\cdot)}{s^2}\right\| _{L^2(B_2)}$ is bounded by a dimensional constant if $0< s<2^{-n-2} e^{-1}n^{-2}$, 
and \eqref{boundedrescalings} follows.

\end{proof}

\subsection{$C^{1,\alpha}$-regularity of the free boundary}

In this section we perform an iteration argument, based on  Proposition \ref{p3}
and Lemma \ref{l6}, that leads to the existence of the unit normal $ \eta_0$ of the free boundary at the origin, 
and provides good decay estimates for $ \Arrowvert \nabla'_{\eta_0} u_r \rVert_{W^{2,2}(B_2)}$.

\par
First we would like to verify that $ u \in \mathscr{B}_\kappa^\varrho( \varepsilon)$ imply that 
$ u_s \in \mathscr{B}^\varrho_\kappa({ \varepsilon})$. It is easy to check that the property of being an 
NTA domain is scaling invariant, in the sense that 
if $ D $ is an NTA domain and $ 0\in \partial D $, then for any $0<s<1$ the set 
$D_s := s^{-1}(D \cap B_s)$ is also an NTA domain with the same parameters as $D$.

\par
Assumption \ref{itm:third} in Definition \ref{d1} holds for $ u_s$ according to Corollary \ref{c2}. Indeed,
 let $ t=s$ in Corollary \ref{c2}, then $ u(sx)=0$ if $ x_n < -1$ .

\par
Thus $u_s$ satisfies  \ref{itm:second}, \ref{itm:third} in Definition \ref{d1}, but it may not
satisfy \ref{itm:forth}.
Instead we consider rescaled solutions defined as follows
\begin{equation} \label{capitalu}
 U_s(x):=\frac{ \omega_n u_s(x)}{\Arrowvert D^3 u_s \rVert_{L^2(B_1)}},
\end{equation}
then assumption \ref{itm:forth} also holds. Indeed, $ \Arrowvert D^3 U_s\rVert_{L^2(B_1)}= \omega_n$ by definition of
$ U_s$, and
\begin{align*}
  \Arrowvert D^3 U_s\rVert_{L^2(B_2)}=
  \frac{\omega_n\Arrowvert D^3 u_s\rVert_{L^2(B_2)}}{\Arrowvert D^3 u_s\rVert_{L^2(B_1)}}=
  \frac{\omega_n \Arrowvert D^3 u \rVert_{L^2(B_{2s})}}{\Arrowvert D^3 u \rVert_{L^2(B_{s})}}\\
  \leq \omega_n \frac{ \omega_n (2s)^\frac{n}{2} + \sigma(\varepsilon)}{\omega_n (s)^\frac{n}{2}- \sigma(\varepsilon)}
  < \kappa,
\end{align*}
according to Lemma \ref{l4} provided $\varepsilon=\varepsilon(n, \kappa, s)$ is small.

 \par
 In the next lemma we show that $ U_s \in \mathscr{B}_\kappa^\varrho ({\gamma \varepsilon}) $ in a normalised coordinate system,
 then we argue inductively to show that $ U_{s^k} \in \mathscr{B}_\kappa^\varrho({\gamma^k \varepsilon})$, $\gamma <1$.

\begin{lemma} \label{l7}
 Assume that $ u \in \mathscr{B}^\varrho_{\kappa}(\varepsilon) $ solves the biharmonic obstacle problem in a 
 normalised coordinate system $ \{e_1, e_2,..., e_n\}$. 
Then for any $0<\alpha <1$ there exist $r_0>0$ and a unit vector $\eta_0  \in \mathbb{R}^n $, such that $ \lvert \eta_0 -e_n \rvert \leq C\varepsilon$,
and for any $ 0<r<r_0$
\begin{equation} \label{ralpha}
\frac{\Arrowvert \nabla'_{\eta_0} u_{r}\rVert_{W^{2,2}(B_2)} }{\Arrowvert D^3 
u_{r}\rVert_{L^2(B_1)} }\leq C r^\alpha \varepsilon,
\end{equation}
provided $ \varepsilon=\varepsilon(n,\kappa,\varrho, \alpha) $ is small enough.
The constant $C>0$ depends only on the given parameters.
\end{lemma}

\begin{proof}
Throughout $ \{e_1,...,e_n \}$ is a fixed coordinate system normalised with respect to the solution 
$ u\in \mathscr{B}_\kappa^\varrho(\varepsilon)$, and $ \nabla'u= \nabla'_{e_n}u$.
We may renormalise the coordinate system with respect to $U_{2s} $ and denote by $ \{e_1^1,...,e_n^1\}$
the set of basis vectors in the new system. Inductively, $ \{e_1^k,...,e_n^k\}$, $ k\in \mathbb{N}$ is 
a normalised system 
with respect to $U_{2s^k}$, and $e_i^0:=e_i$. According to Lemma \ref{l6},
\begin{equation} \label{renorm}
 \lvert e_n^{k+1}-e_n^{k} \rvert \leq C(n) 
 \frac{\Arrowvert \nabla'_{e_n^{k}}  u_{2s^{k+1}}\rVert_{L^2(B_2)}}{ \Arrowvert D^3 u_{s^k }\rVert_{L^2(B_1)}}     \overset{\eqref{boundedrescalings}}{\leq}
C(n) \frac{\Arrowvert \nabla'_{e_n^{k}}  u_{s^{k}}\rVert_{W^{2,2}(B_2)}}{ \Arrowvert D^3 u_{s^k }\rVert_{L^2(B_1)}}   ,
 \end{equation}
provided $ \Arrowvert \nabla'_{e_n^{k}}  U_{s^k} \rVert_{W^{2,2}(B_2)}$ is sufficiently small.
 
\par
In the following discussion $ 0 <s<\tau<1$ are small fixed numbers, satisfying the assumptions in Proposition \ref{p3}.

\par
Now let us consider the sequence of numbers $\{A_k\}_{k\in \mathbb{N}_0}$,
defined as follows:
\begin{equation} \label{Ak+1}
  A_{k}:=\frac{\omega_n \Arrowvert \nabla'_{e_n^{k}} u_{s^k} \rVert_{W^{2,2}(B_2)}}{\Arrowvert D^3 u_{s^k}\rVert_{L^2(B_1)}} , 
  \textrm{ for } k=0,1,2,... .
\end{equation}
By definition, $A_0\leq \varepsilon$, and  
 \begin{equation} \label{A1}
  A_{1}=\frac{\omega_n \Arrowvert \nabla'_{e_n^1} u_{s} \rVert_{W^{2,2}(B_2)}}{\Arrowvert D^3 u_s\rVert_{L^2(B_1)}}=
   \frac{\Arrowvert D^3 u\rVert_{L^2(B_1)}}{\Arrowvert D^3 u_{s}\rVert_{L^2(B_1)}} { \Arrowvert \nabla'_{e_n^1} u_{s} \rVert_{W^{2,2}(B_2)}}.
\end{equation}
Applying Proposition \ref{p3} and Lemma \ref{l6} for the function
$u \in \mathscr{B}_\kappa^\varrho({ \varepsilon }) $, we obtain
\begin{equation} \label{A1+}
\begin{aligned}
 \Arrowvert \nabla'_{e_n^{1}} u_{s}\rVert_{W^{2,2}(B_2)} \leq \tau
  \Arrowvert \nabla'_{e_n^1} u\rVert_{W^{2,2}(B_2)}  \\
 \leq \tau \left(  \Arrowvert \nabla' u\rVert_{W^{2,2}(B_2)}+2|e_n-e_n^1|   \Arrowvert \nabla u\rVert_{W^{2,2}(B_2)}\right) \\
\overset{\eqref{renorm}}{\leq}
   \tau \left(     \Arrowvert \nabla' u\rVert_{W^{2,2}(B_2)}+2C(n)\frac{\Arrowvert \nabla' u\rVert_{W^{2,2}(B_2)}}{|| D^3u | |_{L^2(B_1)}}  \Arrowvert \nabla u\rVert_{W^{2,2}(B_2)} \right)\\
    \overset{\eqref{boundedrescalings}}{\leq}
 \tau C(n,\kappa)   \Arrowvert \nabla' u\rVert_{W^{2,2}(B_2)}.
    \end{aligned}
\end{equation}
Let $ \beta:= \tau C(n,\kappa)$, and $ \beta <\gamma <1$ be fixed numbers. Then
\begin{equation*}
 \frac{\Arrowvert D^3 u\rVert_{L^2(B_1)}}{\Arrowvert D^3 u_{s}\rVert_{L^2(B_1)}}=
 \frac{s^\frac{n}{2} \Arrowvert D^3 u\rVert_{L^2(B_1)}}{\Arrowvert D^3 u \rVert_{L^2(B_{s})}} \\
 \leq
 \frac{s^\frac{n}{2} \omega_n }{s^\frac{n}{2} \omega_n - \sigma( A_0)}
 \leq \frac{\gamma}{\beta}.
\end{equation*}
according to Lemma
\ref{l4}, provided $  A_0 \leq \varepsilon$ is small depending on the parameter $s$ and dimension $n$.
The last inequality together with \eqref{A1} and \eqref{A1+} implies that 
  $ A_1 \leq \gamma \varepsilon$.

 \par
  We use an  induction argument to show that
 \begin{equation} \label{induction}
  A_{k} \leq \gamma^{k} \varepsilon, ~~~ \textrm{ for all } k \in \mathbb{N}_0
 \end{equation}
 for  fixed 1$> \gamma>\beta>\tau >\Lambda_n s>0$. 
 Assuming that \eqref{induction} holds for  $k\in \mathbb{N}$, we will show that $ A_{k+1} \leq \gamma A_k$.

\par
By the induction assumption
\begin{equation}
 \Arrowvert \nabla'_{e_n^{k}} U_{s^k }\rVert_{W^{2,2}(B_2)} = 
 \frac{\omega_n \Arrowvert \nabla'_{e_n^{k}} u_{s^k }\rVert_{W^{2,2}(B_2)} }{\Arrowvert D^3 
 u_{s^k}\rVert_{L^2(B_1)}} 
= A_k \leq  \gamma^k \varepsilon.
\end{equation}
Hence
\begin{equation*}
U_{s^k}= \frac{\omega_n u_{s^k}(x)}{\Arrowvert D^3 u_{s^k}\rVert_{L^2(B_1)}} 
\in \mathscr{B}_\kappa^\varrho({\gamma^k \varepsilon })
\end{equation*}
 in the coordinate system $\{ e_1^{k}, ..., e_n^{k}\}$. By definition, 
$\{ e_1^{k+1}, ..., e_n^{k+1}\} $ is a normalised coordinate system with respect to 
$ U_{2s^{k+1}}\in \mathscr{B}_\kappa^\varrho({\beta^k \varepsilon })$, and  by \eqref{tau}
\begin{equation}
  A_{k+1}=\frac{\omega_n \Arrowvert \nabla'_{e_n^{k+1}} u_{s^{k+1}} \rVert_{W^{2,2}(B_2)}}{\Arrowvert D^3 u_{s^{k+1}}\rVert_{L^2(B_1)}} \leq \frac{\omega_n \tau \Arrowvert \nabla'_{e_n^{k+1}} u_{s^{k}} \rVert_{W^{2,2}(B_2)}}{\Arrowvert D^3 u_{s^{k}}\rVert_{L^2(B_1)}} \frac{\Arrowvert D^3 u_{s^{k}}\rVert_{L^2(B_1)}} {\Arrowvert D^3 u_{s^{k+1}}\rVert_{L^2(B_1)}} .
\end{equation}
First we observe that 
\begin{equation} \label{norm21}
\begin{aligned}
 \frac{\Arrowvert D^3 u_{s^{k}}\rVert_{L^2(B_1)}}{\Arrowvert D^3 u_{s^{k+1}}\rVert_{L^2(B_1)}}=
 \frac{s^\frac{n}{2} \Arrowvert D^3 u_{s^{k}}\rVert_{L^2(B_1)}}{\Arrowvert D^3 u_{s^{k}}\rVert_{L^2(B_{s})}} \\
 \leq
 \frac{s^\frac{n}{2} \omega_n }{s^\frac{n}{2} \omega_n - \sigma(\gamma^{k} A_0)}
 \leq \frac{\gamma}{\beta},
\end{aligned}
\end{equation}
 according to Lemma
\ref{l4}, since $ U_{s^k} \in \mathscr{B}_\kappa^\varrho(\gamma^k \varepsilon)$ and 
$\gamma^{k} \varepsilon <  \varepsilon $ is small.

Next we estimate 
\begin{equation}\label{indklong}
\begin{aligned}
 \Arrowvert \nabla'_{e_n^{k+1}} u_{s^k} \rVert_{W^{2,2}(B_2)}\leq 
  \Arrowvert \nabla'_{e_n^k} u_{s^k} \rVert_{W^{2,2}(B_2)}+2|e_n^{k+1}-e_n^k |  \Arrowvert \nabla u_{s^k} \rVert_{W^{2,2}(B_2)}\\
 \overset{\eqref{renorm}}{ \leq} 
  \Arrowvert \nabla'_{e_n^k} u_{s^k} \rVert_{W^{2,2}(B_2)}+\frac{C(n) \Arrowvert \nabla'_{e_n^k} u_{s^{k}} \rVert_{W^{2,2}(B_2)} }{{\Arrowvert D^3 u_{s^{k}}\rVert_{L^2(B_{1})}} } \Arrowvert \nabla u_{s^k} \rVert_{W^{2,2}(B_2)}\\
 \overset{\eqref{boundedrescalings}} {\leq }
  C(n,\kappa)  \Arrowvert \nabla'_{e_n^k} u_{s^k} \rVert_{W^{2,2}(B_2)}.
  \end{aligned}
\end{equation}

Finally we obtain from  \eqref{norm21} and \eqref{indklong} that
\begin{equation}
 A_{k+1} \leq \frac{ \omega_n \tau C(n,\kappa) \gamma }{\beta} \frac{\Arrowvert \nabla'_{e_n^k} u_{s^k} \rVert_{W^{2,2}(B_2)}}{\Arrowvert D^3 u_{s^k}\rVert_{L^2(B_1)}}= \gamma A_k \leq \gamma^{k+1} A_0,
\end{equation}
this completes the proof of inequality \eqref{induction}.

\par
Next we show that $ \{ e_n^k\}$ is a Cauchy sequence by using \eqref{renorm} and \eqref{induction}. Indeed for any 
$  m, k \in \mathbb{ N }$, 
 \begin{align*}
  \lvert e_n^{k+m}-e_n^k \rvert \leq  \sum_{l=1}^m  \lvert e_n^{k+l}-e_n^{k+l-1} \rvert \leq 
  C(n) \sum_{l=1}^m \Arrowvert \nabla'_{e_n^{k+l-1}}  U_{2s^{k+l}} \rVert_{L^2(B_2)} \\
\leq C(n) \sum_{l=1}^m A_{k+l-1} 
 \leq {C(n) A_0 } \sum_{l=1}^m \gamma^{k+l-1} \leq
\frac {C(n) A_0 }{ (1-\gamma)} \gamma^k,
 \end{align*}
hence $ e_n^k \rightarrow \eta_0 $, as $k\rightarrow \infty$ for some $\eta_0 \in \mathbb{R}^n$, $\lvert \eta_0\rvert=1 $ and
\begin{equation} \label{etaenk}
 \lvert \eta_0 -e_n^k \rvert \leq C'(n) A_0 \gamma^k \leq C'(n) \gamma^k \varepsilon,
\end{equation}
 in particular $ \lvert \eta_0 -e_n \rvert \leq C'(n)\varepsilon$.

 \par
 Now the inequality \eqref{ralpha} follows 
 via a standard iteration argument.
 Let $ 0<\alpha <1$ be any number,  choose  $s=s(n,\alpha) $ small, satisfying the assumption in Proposition \ref{p3}, and such that $\gamma= C_n s <s^\alpha $. 
 If $ 0<r\leq s$, then there exists $ k\in \mathbb{N}_0$, such that $ s^{k+1}\leq r <s^k$. Hence
 \begin{equation} \label{mainineq}
 \begin{aligned}
  \frac{\Arrowvert \nabla'_{\eta_0} u_{r}\rVert_{W^{2,2}(B_2)} }{\Arrowvert D^3 u_{r}\rVert_{L^2(B_1)} }\leq 
 C \frac{\Arrowvert \nabla'_{\eta_0} u_{s^k}\rVert_{W^{2,2}(B_2)} }{\Arrowvert D^3 u_{s^k}\rVert_{L^2(B_1)} } \leq \\
 C \frac{\Arrowvert \nabla'_{e_n^k} u_{s^k}\rVert_{W^{2,2}(B_2)} +2|e_n^k- \eta_0| \Arrowvert \nabla u_{s^k}\rVert_{W^{2,2}(B_2)}}{\Arrowvert D^3 u_{s^k}\rVert_{L^2(B_1)} }
 \leq 
 C \gamma^k \varepsilon\leq  C s^{\alpha k} \varepsilon \leq C r^\alpha \varepsilon,
 \end{aligned}
 \end{equation}
where $C$ depends on the space dimension and on the given parameters.
 
 \end{proof}

 \par
 Now we are ready to prove the $C^{1,\alpha}$-regularity of the free boundary.

\begin{theorem} \label{th1}
Let  $0 <\alpha <1$ be a given number.
 Assume that $  u \in \mathscr{B}_\kappa^\varrho(\varepsilon)$, with an $\varepsilon> 0$
 small, depending on $ \alpha $ and the space dimension. Then there exists $ 0<r_0<1$ depending on the given parameters, such that $\Gamma_u \cap B_{r_0} $ is a  $C^{1,\alpha}$-graph  and the $C^{1,\alpha}$-norm of the graph is bounded by $ C\varepsilon$.
\end{theorem}

\begin{proof}
 Let $ 0<\alpha <1$ and  fix $ s=s(n, \alpha)>0 $ small  as in \eqref{mainineq}.
It follows from Lemma \ref{l7} that for
$ u\in \mathscr{B}^\varrho_{\kappa}(\varepsilon)$
\begin{equation*}
\frac{ \Arrowvert \nabla' u_r \rVert_{W^{2,2}(B_2)}}{\Arrowvert D^3 u_r \rVert_{L^2(B_1)}} \leq C r^\alpha
\rightarrow 0 ~ \textrm{ as } ~ r\rightarrow 0,
 \end{equation*}
 after a change of variable, by choosing
$ e_n = \eta_0 $, where $\eta_0$ is the same vector as in Lemma \ref{l7}.
Then 
\begin{equation*}
\frac{\omega_n  u_r(x)}{\Arrowvert D^3u_r \rVert_{L^2(B_1)}} \rightarrow \frac{1}{6} (x_n)_+^3
\end{equation*}
according to
Lemma \ref{l4}.

\par
So we have shown that in the initial coordinate system,
\begin{equation}
\frac{ \omega_n u(rx)}{ r^3 \Arrowvert D^3u_r\rVert_{L^2(B_1)}} \rightarrow \frac{1}{6} (\eta_0 \cdot x )_+^3 
\textrm{ in } W^{3,2} (B_1)\cap C^{1,\alpha}(B_1), \textrm{ as } r\rightarrow 0,
\end{equation}
 and therefore $ \eta_0 $ is the measure theoretic normal to $\Gamma_u$
at the origin.

\par
 Now let $ x_0 \in \Gamma_u \cap B_{s}$ be a free boundary point, and consider the function 
 $  u_{{x_0,{1/2}}}(x) = \frac{u(x/2+x_0)}{(1/2)^3}, x \in B_2$, then
 \begin{align*}
 U_{x_0}(x):= \frac{\omega_n  u_{{x_0}, 1/2}(x)}{  \Arrowvert D^3 u_{x_0, 1/2}\rVert_{L^2(B_1)}}\in 
 \mathscr{B}^\varrho_\kappa ({ C(n) \varepsilon}).
 \end{align*}
According to Lemma \ref{l7}, $U_{x_0}$ has a unique blow-up
 \begin{equation*}
U_{r, x_0}(x):=\frac{U_{x_0}(rx)}{r^3}= \frac{\omega_n  u_{{x_0}, 1/2}(rx)}{ r^3 \Arrowvert D^3 u_{x_0, 1/2}(rx)\rVert_{L^2(B_1)}}\rightarrow \frac{1}{6}(\eta_{x_0}\cdot x)_+^3 .
 \end{equation*}
 and therefore $\eta_{x_0}$ is the normal to $\Gamma_u$ at $x_0$.
 
 \par
 Next we show that $\eta_x$ is a H\"{o}lder continuous function on $ \Gamma_u \cap B_s$. 
If $ x_0 \in  \Gamma_u \cap B_s$, then $ s^{k+1}< \lvert x_0 \rvert \leq s^k $, for some
$ k\in \mathbb{ N }_0$.
Hence $ \Arrowvert \nabla'_{\eta_0} U_{s^k, x_0}\rVert_{W^{2,2}(B_2)} \leq C \gamma^k \varepsilon$, and
$ \Arrowvert \nabla'_{\eta_{x_0}}U_{r, x_0} \rVert_{W^{2,2}(B_1)} \rightarrow 0$ as $ r\rightarrow 0$.
  Applying Lemma \ref{l7} for the function $ U_{s^k,x_0} \in \mathscr{B}^\varrho_\kappa ({ C \gamma^k \varepsilon})$, we obtain
 \begin{equation} \label{hol}
  \lvert \eta_{x_0} -\eta_0 \rvert
  \leq C \gamma^k \varepsilon
  \leq \frac{C}{\gamma} \lvert x_0 \rvert^{\alpha} \varepsilon.
 \end{equation}
Furthermore, the inequality 
 \begin{equation*}
  \lvert \eta_x -\eta_y \rvert \leq C\lvert x-y \rvert^\alpha  \varepsilon ,  ~~~ \textrm{ for any } x,y \in \Gamma_u
 \end{equation*}
 follows from  \eqref{hol}.

\end{proof}

\section{On the regularity of the solution}

In this section we study the regularity of the solution to the biharmonic obstacle problem. 
Assuming that $ u\in \mathscr{B}_\kappa^\varrho(\varepsilon)$, with $ \varepsilon>0$ small, we derive from Theorem
\ref{th1} that
$ u\in C^{2,1}_{loc} (B_1)$. 
In the end we provide an example showing that without the NTA domain assumption, there exist solutions, that are not 
$C^{2,1}$.

\subsection{ $C^{2,1}$-regularity of the solutions in  $ \mathscr{B}^\varrho_\kappa(\varepsilon)$}
 
\par
After showing the $C^{1,\alpha}$-regularity
of the free boundary $\Gamma_u \cap B_{r_0}$,
 we may go further to derive improved regularity for the solution $ u\in \mathscr{B}_\kappa^\varrho(\varepsilon)$.

\begin{theorem} \label{th2}
 Let $ u\in \mathscr{B}_\kappa^\varrho(\varepsilon)$ be the solution to the biharmonic obstacle problem 
 in $\Omega \supset \supset B_2$, and let $ 0<\alpha <1$ be a fixed number.
 Then there exists $r_0>0$ such that  $u\in C^{2,1}(B_{r_0})$, provided $\varepsilon=\varepsilon(\kappa,\varrho, \alpha)$ is  small. Furthermore, the following estimate holds
 \begin{equation*}
  \Arrowvert u \rVert_{C^{3,\alpha}(\overline{\Omega}_u \cap B_{r_0} )} \leq  
  C(n) \Arrowvert u \rVert_{W^{2,2}(B_2)}\leq C(n) \kappa,
 \end{equation*}
 where $C(n)$ is just a dimensional constant.
\end{theorem}

\begin{proof}

According to Theorem \ref{th1}, $\Gamma_u \cap B_s$ is a graph of a $C^{1,\alpha}$-function. 
We know that $ \Delta u \in W^{1,2}(B_2) $ is a harmonic function in $ \Omega_u:= \{u> 0\}$, and also
$ u \in W^{3,2}(B_2)$, $ u \equiv 0 $ in $ \Omega \setminus \Omega_u$, hence 
$\Delta u =0$ on $\Gamma_u= \partial \Omega_u \cap B_2$ in the trace sense. 
Therefore we may apply Corollary 8.36 in  \cite{GT}, to conclude that
$\Delta u \in C^{1,\alpha} ( (\Omega_u \cup \Gamma_u)  \cap B_{3s/4})$, and 
\begin{equation}
 \Arrowvert \Delta u \rVert_{ C^{1,\alpha} (\overline{ \Omega }_u \cap B_{3s/4}) } \leq C(n) 
 \Arrowvert \Delta u \rVert_{L^\infty( B_1)}.
\end{equation}

\par
It follows from the Calder\'{o}n-Zygmund estimates that $u \in W^{3,p}(B_{s/2})$, for any
$p < \infty$. According to the Sobolev embedding theorem,
$ u \in C^{2,\alpha}(B_{s/2})$, for all $\alpha < 1 $, with the following estimate
\begin{equation} \label{c2alpha}
 \Arrowvert u \rVert_{C^{2,\alpha}(\overline{ \Omega_u } \cap B_{s/2}) } \leq C(n)
 \left(  \Arrowvert \Delta u \rVert_{C^{1,\alpha}(\overline{\Omega}_u \cap B_{3s/4})}+ 
 \Arrowvert u \rVert_{C^{1,\alpha}(B_{3/4})}
 \right).
\end{equation}

\par
Denote by $u_{ij}:= \frac{\partial^2 u}{\partial x_i \partial x_j}$.
Then $ u_{ij} \in W^{1,2}(B_1)\cap C^\alpha(\Omega_u \cap B_{3s/4})$ is a weak solution of 
$ \Delta  u_{ij}= \frac{\partial f_j}{\partial x_i}$ in $ \Omega_u \cap B_{3s/4} $, where
$ f_j:= \frac{\partial \Delta u }{\partial x_j} \in C^\alpha( \Omega_u \cap B_{3/4} )$.
Taking into account that $ u_{ij}=0 $ on $ \partial \Omega_u \cap B_{1/2} $, we may apply 
Corollary 8.36 in  \cite{GT} once again and conclude that 
\begin{equation*}
 \Arrowvert u_{ij}\rVert_{C^{1,\alpha}( \overline{\Omega}_u \cap B_{s/4}) } \leq 
 C(n)\left( \Arrowvert u_{ij}\rVert_{C^{0}( \overline{\Omega}_u \cap B_{s/2})}+
 \Arrowvert \Delta u \rVert_{C^{1,\alpha}( \overline{\Omega}_u\cap B_{3s/4})}
 \right),
\end{equation*}
hence 
\begin{equation*}
 \Arrowvert D^2 u\rVert_{C^{1,\alpha}( \overline{\Omega}_u \cap B_{s/4})} \leq 
 C'_n \left(   \Arrowvert \Delta u \rVert_{C^{1,\alpha}(\overline{\Omega}_u \cap B_{3s/4})}+ 
 \Arrowvert u \rVert_{C^{1,\alpha}(B_{3/4})} \right),
\end{equation*}
according to \eqref{c2alpha}.

\par
Therefore we obtain
\begin{align*}
 \Arrowvert u \rVert_{C^{3,\alpha}(\overline{\Omega}_u \cap B_{s/4} )} \leq  
 \Arrowvert u \rVert_{C^{1,\alpha}(B_{3/4})}+
  \Arrowvert D^2 u\rVert_{C^{1,\alpha}( \overline{\Omega}_u \cap B_{s/4})} \\
  \leq
  C(n) \left(   \Arrowvert \Delta u \rVert_{C^{1,\alpha}(\overline{\Omega}_u \cap B_{3s/4})}+ 
 \Arrowvert u \rVert_{C^{1,\alpha}(B_{3s/4})} \right).
\end{align*}

\par
Taking into account that 
\begin{equation*}
 \Arrowvert D^3 u \rVert_{L^\infty(B_{s/4} ) } \leq  
 \Arrowvert D^3 u\rVert_{C^{0,\alpha}( \overline{\Omega}_u \cap B_{s/4})},
\end{equation*}
we see that $ u\in C^{2,1}( B_{s/4} )$.

\end{proof}

\subsection{In general the solutions are not better than $C^{2,\frac{1}{2}}$}

\par
Let us observe that the assumption $u \in \mathscr{B}_\kappa^\varrho (\varepsilon) $ is essential in the proof
of $u\in C^{2,1}(B_r)$. The next example shows that without our flatness assumptions
there exists a solution to the biharmonic
obstacle problem in $\mathbb{R}^2$, that do not possess $C^{2,1}$- 
regularity.

\begin{example} \label{c21}
 Consider the following function given in polar coordinates in $\mathbb{R}^2$,
 \begin{equation}
  u(r,\varphi)= r^{\frac{5}{2}}\left(  \cos \frac{\varphi}{2} -\frac{1}{5}\cos \frac{5 \varphi}{2} \right), ~~ r \in [0,1), ~~ \varphi \in [-\pi, \pi)
 \end{equation}
then $u \in C^{2, \frac{1}{2}}$ is the solution to the biharmonic zero-obstacle problem in the unit ball $B_1 \subset \mathbb{R}^2$.
\end{example}
\begin{proof}
It is easy to check that $ u \geq 0 $, $ u(x)=0$ if and only if $-1 \leq  x_1 \leq 0$ and $ x_2 =0$. Hence
the set $ \Omega_u = \{u >0\}$ is not an NTA domain, since the complement of  $ \Omega_u $
does not satisfy the corkscrew condition.

\par
Let us show that $ \Delta^2 u $ is a nonnegative measure supported on $[-1,0]\times \{0 \}$.
 For any nonnegative $ f \in C_0^\infty (B_1)$, we compute
 \begin{align*}
  \int_{B_1} \Delta u(x) \Delta f(x)dx =\int_{0}^1 \int_{-\pi}^\pi r  \Delta f(r,\varphi) 
\Delta u(r, \varphi) d \varphi d r \\
 = 6 \int_{0}^1 \int_{-\pi}^\pi r^{\frac{3}{2}}  \Delta f(r,\varphi) 
\cos \frac{\varphi}{2} d \varphi d r 
  =6 \int_0^1 {r^{- \frac{1}{2}}} {f(r,\pi)} dr \geq 0,  
 \end{align*}
where we used integration by parts, and that $ f$ is compactly supported in $ B_1$.

\par
We obtain that $ u $ solves the following variational inequality,
\begin{equation} \label{varin}
 u \geq 0, ~ \Delta^2 u \geq 0, ~ u \cdot \Delta^2 u =0.
\end{equation}
Now we show that $ u$
is the unique minimizer to the following zero-obstacle problem: 
minimize the functional \eqref{J} over 
\begin{align*}
 \mathscr{A}:=\{v \in W^{2,2}(B_1), v \geq 0, \textrm{ s.t. } v= u, 
\frac{\partial v}{\partial n }= \frac{\partial u}{\partial n}, \textrm{ on } \partial B_1 \}
\neq \emptyset.
\end{align*}
According to Lemma \ref{1} there exists a unique minimizer, let us call it $ v$. 
It follows from \eqref{varin}, that
\begin{align*}
 \int_{B_1} \Delta u \Delta (v-u)= \int_{B_1} (v-u) \Delta^2 u  = \int_{B_1} v \Delta^2 u \geq 0.
\end{align*}
Hence 
\begin{align*}
 \int_{B_1} (\Delta u)^2 \leq \int_{B_1} \Delta u \Delta v \leq  
 \left( \int_{B_1} (\Delta u)^2 \right)^\frac{1}{2}   \left( \int_{B_1} (\Delta v)^2 \right)^\frac{1}{2},
\end{align*}
where we used the H\"{o}lder inequality in the last step. 
Therefore we obtain
\begin{align*}
  \int_{B_1} (\Delta u)^2 \leq  \int_{B_1} (\Delta v)^2, 
\end{align*}
thus $ u \equiv v$, and $u$ solves the biharmonic zero-obstacle problem in the unit ball.

\par
However
 $ \Delta u= 6 r^{\frac{1}{2}} \cos \frac{\varphi}{2}$, which implies that $u$ is
 $ C^{2,\frac{1}{2}}$, and that the exponent $ \frac{1}{2}$ is optimal, in particular $u$ is not 
 $ C^{2,1}$.

\end{proof}

\appendix

\section {Estimates on derivatives of biharmonic functions }

In this part of the paper estimates on derivatives of biharmonic functions are obtained. 
We believe that these estimates are known, but we could not find a reference, and therefore included in the paper.

\begin{lemma}
Let $v$ be a biharmonic function in the ball $ B_1 \subset \mathbb{R}^n$, and assume that $B_r(x_0)\subset B_1$. Then
\begin{equation} \label{Tinduction}
|D^\alpha v (x_0)|\leq \frac{(2^{n+1}nk)^k }{|B_1| r^{n+k}} \left( \| v \|_{L^1(B_r(x_0))}+
\frac{r^2}{2(n+2)}\| \Delta v \|_{L^1(B_r(x_0))} \right),
\end{equation}
where $ \alpha$ is a multiindex, and $k=|\alpha|$.
\end{lemma}
\begin{proof}
The following mean value properties are known for biharmonic functions
\begin{equation}
\label{MVP_sphere}
v(x_0)=\fint_{\partial B_r(x_0)} vdS -\frac{r^2}{2n} \Delta v(x_0)
\end{equation}
and
\begin{equation}
\label{MVP_ball}
v(x_0)=\fint_{B_r(x_0)} vdx -\frac{r^2}{2(n+2)} \Delta v(x_0).
\end{equation}
The proofs of \eqref{MVP_sphere} and \eqref{MVP_ball} are similar to the proofs of the mean value properties for harmonic functions.  For a fixed $x_0$, let $ \phi(r):=\fint_{\partial B_r(x_0)} vdS$. It is easy to see  by Green's formula that 
\begin{equation}
\label{derMVP}
 \phi^\prime (r)=\frac{r}{n} \fint_{B_r(x_0)}\Delta v dx=
\frac{r}{n} \Delta v(x_0),
\end{equation}
 where we also used the mean value property for the harmonic function
 $\Delta v$.  Hence \eqref{MVP_sphere} follows by integrating \eqref{derMVP} in the interval $(0,r)$.

\par
Now \eqref{MVP_ball} can be shown by using \eqref{MVP_sphere}  and the  co-area formula;
\begin{equation*}
\begin{aligned}
\fint_{B_r(x_0)}vdx=|B_r|^{-1}\int_0^r \int_{\partial B_s(x_0)} vdS ds=
\int_0^r \frac{n s^{n-1}}{r^n} \left(v(x_0)+\frac{s^2}{2n} \Delta v(x_0)  \right) ds \\
=v(x_0) \frac{n}{r^n} \int_0^rs^{n-1} ds+\Delta v(x_0)\frac{ 1}{2r^n }\int_0^rs^{n+1} ds
=v(x_0)+\Delta v(x_0)\frac{ r^2 }{2(n+2)}.
\end{aligned}
\end{equation*}

\par
Let us proceed to the proof of \eqref{Tinduction}.  Estimate \eqref{Tinduction} is well known  for harmonic functions, which will be used to show that it also holds for biharmonic functions.
We follow the proof of estimates on derivatives of harmonic functions (see for instance \cite{evans}), and employ \eqref{MVP_sphere} and \eqref{MVP_ball}.
The proof uses an argument of induction on $ k=| \alpha|$.
The formula \eqref{MVP_ball} implies that 
\begin{equation}
\label{ }
|v(x_0)|\leq \frac{1}{r^n |B_1|}\left( \|v\|_{L^1(B_r(x_0))}+ \frac{r^2}{2(n+2)}  \|\Delta v\|_{L^1(B_r(x_0))} \right).
\end{equation}
Let $ k=1$, then 
\begin{equation}
\begin{aligned}
v_{x_i}(x_0)=\fint_{B_{r/2}(x_0)} v_{x_i}dx -\frac{r^2}{2^3(n+2)} \fint_{B_{r/2}(x_0)}\Delta v_{x_i} dx \\
=\frac{2^n}{|B_1| r^n} \int_{\partial B_{r/2}(x_0)} v{\nu_i}dS -\frac{r^22^{n-2}}{2(n+2)|B_1| r^n} \int_{\partial B_{r/2}(x_0)}{\nu_i} \Delta vdS.
\end{aligned}
\end{equation}
hence
\begin{equation}
\begin{aligned}
|v_{x_i}(x_0)|\leq \frac{2n \|v\|_{L^\infty(\partial B_{r/2}(x_0))}}{ r} +
\frac{rn \| \Delta v\|_{L^\infty(\partial B_{r/2}(x_0))}}{2^2(n+2) } \\
 \leq 
 \frac{ 2^{n+1}n \|v\|_{L^1( B_{r}(x_0))}}{ |B_1|  r^{n+1}}
 + \frac{n 2^{n-1}r^2\|\Delta v\|_{L^1( B_{r}(x_0))}}{ 2(n+2)|B_1|  r^{n+1}}  
 +\frac{r^2 n 2^{n-1}\|\Delta v\|_{L^1(B_{r}(x_0))}}{2(n+2) |B_1|  r^{n+1}} \\
 \leq 
 \frac{n2^{n+1}}{|B_1|  r^{n+1}}\left( \|v\|_{L^1(B_{r}(x_0))}
 +\frac{r^2}{ 2^2(n+2)}\|\Delta v\|_{L^1(B_{r}(x_0))} \right).
 \end{aligned}
\end{equation}
Assuming that \eqref{Tinduction} is true for $ k-1$, we will show that it is true for $k$. Let $|\alpha|=k$, and 
$ D^\alpha v=(D^\beta v)_{x_i} $, where $ |\beta |=k-1$.
By \eqref{MVP_ball}
\begin{equation}
\begin{aligned}
D^\alpha v(x_0)= \fint_{B_{r/k}(x_0)} D^\alpha v dx -\frac{r^2}{k^2 2(n+2)} D^\alpha \Delta v(x_0)  \\
=\frac{k^n}{|B_1|  r^n} \int_{\partial B_{r/k}(x_0)} D^\beta v {\nu_i}dS -
\frac{r^2}{k^2 2(n+2)} D^\alpha \Delta v(x_0)
\end{aligned}
\end{equation}
Hence
\begin{equation}
\begin{aligned}
|D^\alpha v(x_0)| \leq
\frac{kn}{r}  \|D^\beta v \|_{L^\infty(\partial B_{r/k(x_0)})} +
\frac{r^2}{k^2 2(n+2)}| D^\alpha \Delta v(x_0)|.
\end{aligned}
\end{equation}
If $ x \in \partial B_{r/k}(x_0)$, then $ B_{r(k-1)/k}(x) \subset B_r(x_0)$, and by induction assumption
\begin{equation}
\begin{aligned}
|D^\alpha v (x_0)| 
\leq \frac{kn (k-1)^{k-1} (2^{n+1}n)^{k-1}k^{n+k-1}}{r |B_1|  r^{n+k-1} (k-1)^{n+k-1}} \| v\|_{L^1(B_{r}(x_0))} \\
+\frac{ r^2(k-1)^2kn (k-1)^{k-1} (2^{n+1}n)^{k-1}k^{n+k-1}}{k^2 2(n+2)r |B_1|  r^{n+k-1} (k-1)^{n+k-1}} 
\|\Delta v\|_{L^1(B_{r}(x_0))} \\
+\frac{r^2(2^{n+1}nk)^k}{|B_1|  r^{n+k}k^22(n+2)}
\|\Delta v\|_{L^1(B_{r}(x_0))} \leq 
\frac{(kn2^{n+1})^k}{|B_1|  r^{n+k}}\| v\|_{L^1(B_{r}(x_0))} \\
+\frac{(kn2^{n+1})^k((k-1)^2+1)}{2(n+2)|B_1| k^2 r^{n+k-2}}\| \Delta v\|_{L^1(B_{r}(x_0))}\\
\leq \frac{(kn2^{n+1})^k}{|B_1|  r^{n+k}}
\left( \| v\|_{L^1(B_{r}(x_0))}+\frac{r^2}{2(n+2)}   \|\Delta  v\|_{L^1(B_{r}(x_0))}\right).
\end{aligned}
\end{equation}
\end{proof}

\addcontentsline{toc}{section}{\numberline{}References}

\bibliographystyle{plain}
\bibliography{biharmonic}

\end{document}